\documentclass[reqno,tbtags,a4paper]{amsart}

\usepackage{amsmath,amssymb,amsthm,amsfonts,amscd}
\usepackage{mathrsfs}

\newcommand{\R}{\mathbb{R}}

\newcommand{\C}{\mathbb{C}}
\newcommand{\Q}{\mathbb{Q}}
\newcommand{\Z}{\mathbb{Z}}

\newcommand{\supp}{\mathop{\mathrm{supp}}\nolimits}

\newcommand{\T}{\mathcal{T}}
\newcommand{\LL}{\mathcal{L}}
\newcommand{\m}{\mathbf{m}}
\newcommand{\codim}{\mathop{\mathrm{codim}}\nolimits}

\newcommand{\D}{\mathbf{D}}

\newtheorem{theorem}{Theorem}[section]

\newtheorem{cor}{Corollary}[section]
\newtheorem{lem}{Lemma}[section]
\newtheorem*{SKL}{Sabitov's Key Lemma}

\theoremstyle{definition}

\newtheorem{defin}{Definition}[section]
\newtheorem{remark}{Remark}[section]
\newtheorem{quest}{Question}[section]

\author{Alexander A. Gaifullin}

\thanks{The work was partially supported by the Russian Foundation for Basic Research (projects 10-01-92102 and 11-01-00694), by a grant of the President of the Russian Federation (project NSh-5413.2010.1), by a grant from Dmitri Zimin's ``Dynasty'' foundation and by a programme of the Branch of Mathematical Sciences of the Russian Academy of Sciences.}

\title{Sabitov polynomials for volumes of polyhedra in four dimensions}

\date{}

\address{Steklov
Mathematical Institute (Moscow), Moscow State University, \textit{and} Institute for Information Transmission Problems (Moscow)}

\email{gaifull@higeom.math.msu.su}

\begin{document}

\begin{abstract}
In 1996 I.\,Kh.~Sabitov proved that the volume of a simplicial polyhedron in a $3$-dimensional Euclidean space is a root of certain polynomial with coefficients depending on the combinatorial type  and on edge lengths of the polyhedron only. Moreover, the coefficients of this polynomial are polynomials in edge lengths of the polyhedron. This result implies that the volume of a simplicial polyhedron with fixed combinatorial type and edge lengths can take only finitely many values. In particular, this yields that the volume of a flexible polyhedron in a $3$-dimensional Euclidean space is constant. Until now it has been unknown whether these results can be obtained in dimensions greater than~$3$. In this paper we prove that all these results hold for polyhedra in a $4$-dimensional Euclidean space. 
\end{abstract}

\maketitle

\section{Introduction}\label{section_intro}

In 1996 I.\,Kh.~Sabitov~\cite{Sab96} proved that the volume of a (not necessarily convex) simplicial polyhedron in~$\R^3$ is a root of certain polynomial whose coefficients are polynomials in the squares of edge lengths of the polyhedron. To give rigorous formulation of this result we need to specify what is understood under a simplicial polyhedron. Let $K$ be a triangulation of a closed oriented surface. Then an \textit{oriented polyhedron} (or a \textit{polyhedral surface}\/) of combinatorial type~$K$ is a mapping $P:K\to \R^3$ whose restriction to every simplex of~$K$ is linear. Notice that we do not require $P$ to be an embedding, so we allow a polyhedron to be degenerate and self-intersected.

By $[uvw]$ we denote the oriented triangle of~$K$ with vertices~$u$, $v$, and~$w$ and the orientation given by the prescribed order of vertices. Then the triangles $[uvw]$, $[vwu]$, and $[wuv]$ coincide and $[vuw]$ is the same triangle with the opposite orientation. An oriented triangle~$[uvw]$ is said to be positively oriented if its orientation coincides with the given orientation of~$K$. The set of all positively oriented triangles of~$K$ will be denoted by~$K_+$. For points $A_0,A_1,A_2,A_3\in\R^3$, we denote by~$[A_0A_1A_2A_3]$ the oriented tetrahedron with vertices $A_0,A_1,A_2,A_3$ and by $V([A_0A_1A_2A_3])$ its oriented volume. (Notice that the tetrahedron~$[A_0A_1A_2A_3]$ may be degenerate.) 

Choose an arbitrary base point $O\in\R^3$. The \textit{generalized volume} of a polyhedron~$P:K\to\R^3$ is defined by 
\begin{equation}\label{eq_vol3}
V(P)=\sum_{[uvw]\in K_+}V([O\,P(u)P(v)P(w)]).
\end{equation}
Obviously, the generalized volume of~$P$ is independent of the choice of the base point $O\in\R^3$.

\begin{remark}
If $P:K\to\R^3$ is an embedding, then $V(P)$ coincides up to sign with the usual volume of the domain in~$\R^3$ bounded by the surface~$P(K)$. So in a general case the generalized volume~$V(P)$ should be understood as the volume of some ``generalized domain'' bounded by the singular surface~$P(K)$. This can be made rigorous in the following way. We have
\begin{equation}\label{eq_vol3m}
V(P)=\int_{\R^3}\lambda(x)\,dx,
\end{equation}
where $dx=dx_1dx_2dx_3$ is the standard measure in~$\R^3$ and $\lambda(x)$ is the algebraic intersection number of an arbitrary curve in~$\R^3$ going from~$x$ to the infinity and the oriented singular surface~$P(K)$. (Obviously, $\lambda$ is a piecewise constant function defined off the surface~$P(K)$.)
Thus a polyhedron $P:K\to\R^3$ should be thought about as a three-dimensional object ``bounded by the singular surface~$P(K)$'' rather than as a two-dimensional object.
\end{remark}

If $[uv]$ is an edge of~$K$, then we denote by~$l_{uv}(P)$ the square of the length of the edge~$[P(u)P(v)]$ of a polyhedron~$P:K\to\R^3$. By $l(P)$ we denote the set of all squared edge lengths~$l_{uv}(P)$, where $[uv]\in K$.

Recall that a polynomial is called \textit{monic\/} if its leading coefficient is equal to~$1$.

\begin{theorem}[I.\,Kh.~Sabitov~\cite{Sab96}]\label{theorem_main3}
Let $K$ be a triangulation of a closed oriented surface. Then there exists a monic  polynomial 
$$
Q(V,l)=V^{N}+a_1(l)V^{N-1}+a_2(l)V^{N-2}+\cdots+a_N(l)
$$
such that $a_j\in\mathbb{Q}[l]$ and for every oriented polyhedron $P:K\to\mathbb{R}^3$, the generalized volume~$V(P)$ satisfies the equation~$Q(V(P),l(P))=0$. Here $l$ denotes the set of variables~$l_{uv}=l_{vu}$ corresponding to edges~$[uv]\in K$.\end{theorem}

The polynomial $Q$ is called a \textit{Sabitov polynomial} for polyhedra of combinatorial type~$K$. The proof of this theorem is given in~\cite{Sab96}--\cite{Sab98b}. This proof is constructive, that is, it gives an explicit procedure for finding the polynomial~$Q$.  In~\cite{CSW97} R.~Connelly, I.\,Kh.~Sabitov, and A.~Walz modified the initial proof of I.\,Kh.~Sabitov so that it became less complicated, though less elementary and non-constructive. For a survey of related results, see~\cite{Sab11}. 

One of the most important applications of  Theorem~\ref{theorem_main3} is its application to the so-called \textit{Bellows Conjecture\/}. A \textit{flex} of a polyhedron $P:K\to\R^3$ is a continuous family of polyhedra $P_t:K\to\R^3$, $0\le t\le 1$, such that $P_0=P$, the edge lengths of the polyhedra~$P_t$ are constant, and the polyhedra $P_{t_1}$ and $P_{t_2}$ are not congruent unless $t_1=t_2$. A polyhedron~$P$ is called \textit{flexible\/} if it admits a flex. The famous Cauchy Theorem asserts that no convex polyhedron is flexible. In 1897 R.~Bricard~\cite{Bri97} constructed his famous flexible octahedra and, moreover, classified all flexible octahedra and found out that all they are self-intersected. The first example of a flexible embedded polyhedron was obtained by R.~Connelly~\cite{Con77}. The Bellows Conjecture asserts that the generalized volume~$V(P_t)$ of a flexible polyhedron is constant. This conjecture was posed in~\cite{Kui78} and~\cite{Con78} in 1978.  The name ``Bellows Conjecture'' is due to R.~Connelly.  Theorem~\ref{theorem_main3} easily implies the following corollary.

\begin{cor}\label{cor_flex3}
The Bellows Conjecture holds, i.\,e., for any flex $P_t:K\to\R^3$, $0\le t\le 1$, the generalized volume~$V(P_t)$ is constant.
\end{cor}

Until now it has been unknown whether Sabitov polynomials  for arbitrary simplicial polyhedra exist in dimensions $n>3$. The only partial result is the remark of I.\,Kh.~Sabitov~\cite{Sab11} that the Sabitov polynomials exist for pyramids, where a simplicial polyhedron is called a \textit{pyramid\/} if it contains a vertex that is joined by edges with all other vertices of the polyhedron.

In this paper our goal is to obtain analogues of Theorem~\ref{theorem_main3} and Corollary~\ref{cor_flex3} for polyhedra in the $4$-dimensional Euclidean space~$\R^4$. First we need to say what is understood under a polyhedron in~$\R^n$. A natural analogue of the above notion of a polyhedron in~$\R^3$ is obtained by replacing the triangulated oriented surface~$K$ with an arbitrary oriented $(n-1)$-dimensional pseudomanifold.

\begin{defin}[Pseudomanifold polyhedron]
A finite simplicial complex~$K$ is called a \textit{$k$-dimensional pseudomanifold\/} if every simplex of~$K$ is contained in a $k$-simplex and every $(k-1)$-simplex of~$K$ is contained in exactly two $k$-simplices. A $k$-dimensional pseudomanifold~$K$ is said to be \textit{oriented\/} if all $k$-simplices of~$K$ are endowed with orientations so that, for every $(k-1)$-simplex~$\tau$, the orientations of the two $k$-simplices containing~$\tau$ induce opposite orientations of~$\tau$.  
Let $K$ be an oriented $(n-1)$-dimensional pseudomanifold. An \textit{oriented pseudomanifold polyhedron\/} of combinatorial type~$K$ is a mapping $P:K\to\R^n$ whose restriction to every simplex of~$K$ is linear. 
\end{defin}

Though this definition covers most of interesting examples, we shall need a more general notion of polyhedron, which will be referred as \textit{cycle polyhedron\/}. This definition of polyhedron  appeared in the thesis by A.\,L.~Fogelsanger~\cite{Fog88}.

Let~$\Delta^M$ be an $M$-dimensional simplex. Let $\mathscr{C}_k(\Delta^M)$ denote the $k$th simplicial chain group of~$\Delta^M$ with integral coefficients. Recall that $\mathscr{C}_k(\Delta^M)$ is the free Abelian group of rank~$\binom{M+1}{k+1}$ generated by oriented $k$-faces~$\sigma$ of~$\Delta^M$ with relations $\bar\sigma=-\sigma$, where $\bar\sigma$ is the face~$\sigma$ with the orientation reversed. Let  $\partial: \mathscr{C}_k(\Delta^M)\to \mathscr{C}_{k-1}(\Delta^M)$ be the boundary operator. It is well known that $\partial^2=0$ and the homology groups of the chain complex~$\mathscr{C}_*(\Delta^M)$ are trivial in positive dimensions and~$\Z$ in dimension~$0$. A chain~$Z\in \mathscr{C}_k(\Delta^M)$ is called a \textit{cycle} if $\partial Z=0$. 
The \textit{support\/} of a chain~$Z\in \mathscr{C}_k(\Delta^M)$ is the simplicial subcomplex $\supp(Z)\subset\Delta^M$ consisting of all $k$-simplices $\sigma\subset\Delta^M$ entering into~$Z$ with nonzero coefficients and all their subsimplices.
Notice that the support of a $k$-chain is always a \textit{pure} $k$-dimensional simplicial complex. This means that every simplex of this complex is contained in a $k$-dimensional simplex of this complex. 
In the sequel we as a rule do not mention the simplex~$\Delta^M$ and say that $Z$ is a $k$-dimensional cycle if $Z$ is an element of~$\mathscr{C}_k(\Delta^M)$  for some simplex~$\Delta^M$ and $\partial Z=0$. The number~$M$ is irrelevant since we can always embed~$\Delta^M$ as a face of~$\Delta^{M'}$ for any~$M'>M$. Actually, it is convenient to suppose that all cycles under consideration are simplicial cycles in some big simplex~$\Delta^M$ and all simplicial complexes under consideration are subcomplexes of~$\Delta^M$.  All cycles under consideration will be cycles with integral coefficients.   

\begin{defin}[Cycle polyhedron]
An \textit{oriented cycle polyhedron} in~$\R^n$ is a pair~$(Z,P)$ such that $Z$ is an $(n-1)$-dimensional cycle and $P:\supp(Z)\to \R^n$ is a mapping whose restriction to every simplex of $\supp(Z)$ is linear. We shall say that a polyhedron~$(Z,P)$ has combinatorial type~$Z$ and we shall write $P:Z\to\R^n$ to indicate that $P$ is an oriented polyhedron of combinatorial type~$Z$.
The images of vertices and edges of~$\supp(Z)$ under the mapping~$P$ are called vertices and edges of the polyhedron respectively.
\end{defin}

Every pseudomanifold polyhedron can be regarded as a cycle polyhedron. Indeed, to transform a pseudomanifold polyhedron $P:K\to \R^n$ into a cycle polyhedron we just take the fundamental cycle~$[K]$ of the oriented pseudomanifold~$K$. 
In the sequel we always work with cycle polyhedra and omit the word ``cycle''.

Any $(n-1)$-dimensional cycle $Z$ can be written as
\begin{equation}\label{eq_cycl}
Z=\sum_{i=1}^kq_i\left[v_1^{(i)}\ldots v_n^{(i)}\right],
\end{equation}
where $q_i\in\Z$ and~$\bigl[v_1^{(i)}\ldots v_n^{(i)}\bigr]$ are  oriented $(n-1)$-simplices of $\supp(Z)$. Choose an arbitrary base point $O\in\R^n$.
The \textit{generalized volume} of an oriented polyhedron~$(Z,P)$ is defined by 
\begin{equation}\label{eq_volZ}
V_Z(P)=\sum_{i=1}^kq_i\, V\!\left(\left[O\,P\bigl(v_1^{(i)}\bigr)\ldots P\bigl(v_n^{(i)}\bigr)\right]\right).
\end{equation}
The following lemma is straightforward.

\begin{lem}
If $\partial Z=0$, then the generalized volume~$V_Z(P)$ is independent of the choice of the base point~$O\in\R^n$.
\end{lem}

Now we are ready to formulate the main result of this paper.

\begin{theorem}\label{theorem_main4}
Let $Z$ be a $3$-dimensional cycle. Then there exists a monic polynomial 
\begin{equation}\label{eq_Sab}
Q(V,l)=V^{N}+a_1(l)V^{N-1}+a_2(l)V^{N-2}+\cdots+a_N(l)
\end{equation}
such that $a_j\in\mathbb{Q}[l]$ and for every polyhedron $P:Z\to\mathbb{R}^4$, the generalized volume~$V_Z(P)$ satisfies the equation~$Q(V_Z(P),l(P))=0$. Here $l$ denotes the set of variables~$l_{uv}=l_{vu}$ corresponding to edges~$[uv]\in \mathrm{supp}(Z)$.\end{theorem}

The polynomial~$Q$ will be called a \textit{Sabitov polynomial} for polyhedra of combinatorial type~$Z$. As in dimension~$3$, a \textit{flex\/} of an oriented polyhedron $P:Z\to\R^n$ is a continuous family of oriented polyhedra $P_t:Z\to\R^n$, $0\le t\le 1$, such that $P_0=P$, the edge lengths of the polyhedra~$P_t$ are constant, and the polyhedra $P_{t_1}$ and $P_{t_2}$ are not congruent unless $t_1=t_2$. 
An interesting class of flexible $4$-dimensional  cross-polytopes was constructed by A.~Walz in 1998. Further examples were obtained by H.~Stachel~\cite{Sta00}.
Theorem~\ref{theorem_main4} immediately implies the following

\begin{cor}\label{cor_flex4}
Let $P_t:Z\to \mathbb{R}^4$, $0\le t\le 1$, be a flex of an oriented polyhedron. Then the generalized volume $V_Z(P_t)$ is constant.
\end{cor}

\begin{remark}\label{remark_volumeint4}
If $P:K\to\R^n$ is an oriented pseudomanifold polyhedron, then formula~\eqref{eq_volZ} turns into the standard formula
\begin{equation*}
V(P)=\sum_{[v_1\ldots v_n]\in K_+}V([O\,P(v_1)\ldots P(v_n)]),
\end{equation*}
which is a direct analogue of~\eqref{eq_vol3}. We can also generalize formula~\eqref{eq_vol3m} to cycle polyhedra of arbitrary dimension. We have $V_{Z}(P)=\int_{\R^n}\lambda(x)dx$, where $dx=dx_1\ldots dx_n$ is the standard measure in~$\R^n$ and $\lambda(x)$ is the algebraic intersection number of an arbitrary curve going from~$x$ to the infinity and the singular cycle~$P(Z)$.
\end{remark}
\begin{remark}
Let $K$ be a pure $(n-1)$-dimensional simplicial complex. 
It is easy to see that $(n-1)$-dimensional cycles~$Z$ whose supports are contained in~$K$ are in one-to-one correspondence with homology classes~$z\in H_{n-1}(K,\Z)$. The following point of view is useful. We shall say that a mapping $P:K\to \R^n$ whose restriction to every simplex of~$K$ is linear is a \textit{polyhedron of combinatorial type}~$K$. Different homology classes~$z\in H_{n-1}(K,\Z)$ (equivalently, different $(n-1)$-cycles $Z$ such that $\supp(Z)\subset K$) can be regarded as different \textit{orientations\/} of the polyhedron~$P:K\to\R^n$. To each such orientation is assigned the generalized volume~$V_Z(P)$ of the oriented polyhedron~$(Z,P)$. 
\end{remark}

\begin{remark}
Until now we have spoken on simplicial polyhedra only.
However, we can  generalize our results to non-simplicial polyhedra. An arbitrary polyhedron in~$\R^4$ can be defined, for instance, in the following way. First, we define a \textit{polytopal cell complex\/} as a cell complex consisting of combinatorial convex polytopes glued along combinatorial equivalences of their facets. (A \textit{combinatorial convex polytope} is a convex polytope up to combinatorial equivalence, that is, a convex polytope with forgotten Euclidean structure.) Now an \textit{oriented polytope\/} in~$\R^4$ is a triple $(K,Z,P)$ such that $K$ is a pure $3$-dimensional polytopal cell complex, $Z\in\mathscr{C}_3(K,\Z)$ is a cellular cycle, and $P:K\to\R^4$ is a mapping that realizes every cell of~$K$ as a convex polytope in~$\R^4$. The volume of this polyhedron can be defined by the formula $V_{Z}(P)=\int_{\R^n}\lambda(x)dx$ as in Remark~\ref{remark_volumeint4}.
A ``naive'' attempt to generalize Theorem~\ref{theorem_main4} and Corollary~\ref{cor_flex4} to non-simplicial polyhedra certainly fails. Indeed, one can easily find a deformation of the standard cube~$[0,1]^4$ such that the edge lengths are constant while the volume varies continuously. However, the following result holds. \textit{The volume~$V_Z(P)$ of an arbitrary polyhedron~$(K,Z,P)$ in~$\R^4$ is a root of certain monic polynomial $Q$ determined solely by the combinatorial structure of~$(K,Z)$ and the metrics of all faces of~$P(K)$. If $P_t:Z\to \mathbb{R}^4$, $0\le t\le 1$, is a flex of an oriented polyhedron $P_0=P$, that is, a deformation preserving the metrics of all faces of~$P(K)$, then the volume $V_Z(P_t)$ is constant\/}. These assertions follow immediately from Theorem~\ref{theorem_main4} and Corollary~\ref{cor_flex4}, since we can subdivide an arbitrary polyhedron so as to obtain a simplicial polyhedron. This reasoning mimics the reasoning of I.\,Kh.~Sabitov~\cite{Sab96}--\cite{Sab98b} in dimension~$3$.
\end{remark}

It is useful to give an algebraic reformulation of Theorem~\ref{theorem_main4}. In dimension~$3$ a similar reformulation has been given in~\cite{CSW97}.
Let $K$ be a finite pure $(n-1)$-dimensional simplicial complex with $m$ vertices.
For any polyhedron $P:K\to\R^n$ and any vertex $v\in K$, we denote by $x_{v,i}(P)$ the $i$th coordinate of the point $P(v)\in\R^n$. A polyhedron $P:K\to\R^n$ is uniquely determined by $m$ points $P(v)\in\R^n$, which can be chosen arbitrarily independently of each other. So it is natural to regard the $mn$ coordinates~$x_{v,i}$ of these $m$ points as independent variables. By $x_K$ we denote the set of these $mn$ independent variables and by $\Z[x_K]$ and $\Q[x_K]$ we denote the polynomial rings in these $mn$ variables with integral and rational coefficients respectively. 
The $n$-tuples $(x_{v,1},\ldots,x_{v,n})$ of independent variables are called \textit{universal vertices\/} of the \textit{universal polyhedron\/} of combinatorial type~$K$. 
For every edge~$[uv]\in K$, the \textit{square of the universal length\/} of it is defined to be the element
$l_{uv}\in\Z[x_K]$ given by
\begin{equation}\label{eq_lx}
l_{uv}=\sum_{i=1}^n(x_{u,i}-x_{v,i})^2.
\end{equation}
If two vertices~$u,v\in K$ are not connected by an edge in~$K$, the same formula~\eqref{eq_lx} gives the \textit{square of the universal length} of the diagonal~$[uv]$. 
Indeed, it is convenient to say that a \textit{universal $n$-dimensional point\/} is an $n$-tuple $x=(x_1,\ldots,x_n)$, where $x_i$ are independent variables, and the \textit{square of the universal distance\/} between two universal points $x=(x_1,\ldots,x_n)$ and $y=(y_1,\ldots,y_n)$ is given by $l_{xy}=\sum_{i=1}^n(x_i-y_i)^2$.

If $Z$ is an $(n-1)$-dimensional cycle  with support contained in~$K$, then the \textit{universal volume\/} of the \textit{universal oriented polyhedron\/} of combinatorial type~$Z$ is the element $V_Z\in\Q[x_K]$ given by
\begin{equation}
\label{eq_VZ}
V_Z=\frac{1}{n!}\sum_{i=1}^kq_i
\left|
\begin{array}{cccc}
x_{v_1^{(i)},1}&x_{v_1^{(i)},2}&\cdots & x_{v_1^{(i)},n}\\
x_{v_2^{(i)},1}&x_{v_2^{(i)},2}&\cdots & x_{v_2^{(i)},n}\\
\vdots & \vdots & \ddots & \vdots\\
x_{v_n^{(i)},1}&x_{v_n^{(i)},2}&\cdots & x_{v_n^{(i)},n}
\end{array}
\right|
\end{equation}
where $Z$ is given by~\eqref{eq_cycl}.

Now, for any polyhedron $P:K\to\R^n$, one can substitute the coordinates~$x_{v,i}(P)$ of the points~$P(v)$ for the variables~$x_{v,i}$. Then the values of the universal squared edge lengths~$l_{uv}$ will become equal to the squared edge lengths~$l_{uv}(P)$ of the polyhedron~$P$ and the value of the universal volume~$V_Z$ will become equal to the generalized volume~$V_Z(P)$. 

Let $R_K\subset\Q[x_K]$ be the $\Q$-subalgebra generated by all elements~$l_{uv}$ such that $[uv]$ is an edge of~$K$. 
The algebraic reformulation of Theorem~\ref{theorem_main4} is as follows.  

\begin{theorem}\label{theorem_alg4}
Let $Z$ be a $3$-dimensional cycle and let $K$ be its support. Then the element $V_Z\in\Q[x_K]$ is integral over the ring~$R_K$, i.\,e., there exists a monic polynomial $Q\in R_K[V]$ such that~$Q(V_Z)=0$.
\end{theorem}

\begin{remark}
Theorem~\ref{theorem_main3} is usually formulated so that the Sabitov polynomial~$Q$ is required to be even. It is reasonable because reversing the orientation of a polyhedron we reverse the sign of its generalized volume. We also could formulate Theorem~\ref{theorem_main4} with this additional requirement. However, thus we would obtain an equivalent statement because the assertions that the elements~$V_Z$ and~$V_Z^2$ are integral over the ring~$R_K$ are equivalent to each other.
\end{remark}

As it has already been mentioned above, a polyhedron $P:K\to \R^n$  is uniquely determined by the images of vertices of~$K$. Hence it is often convenient to regard $P$ as a mapping $S\to\R^n$, where $S$ is the vertex set of~$K$. The next useful step was made in~\cite{CSW97} (in dimension~$3$).
Consider an arbitrary field~$F$ and say that an  \textit{oriented polyhedron\/} in~$F^n$ is a pair $(Z,P)$ such that $Z$ is an $(n-1)$-dimensional cycle and $P:S\to F^n$ is a mapping, where $S$ is the vertex set of the support of~$Z$. 
By definition, the squared edge lengths~$l_{uv}(P)$ and the generalized volume~$V_Z(P)$ are obtained by substituting the coordinates of the points~$P(v)$ in the right-hand sides of~\eqref{eq_lx} and~\eqref{eq_VZ} respectively. (The generalized volume~$V_Z(P)$ is well defined only if the characteristic of the field~$F$ is not $2,3,\ldots,n$.) 
It follows from Theorem~\ref{theorem_alg4} that Theorem~\ref{theorem_main4} still holds if we replace the field~$\R$ with an arbitrary field~$F$ of characteristic not equal to $2,3,\ldots,n$.

Actually, we shall need the cases $F=\R$ and $F=\C$ only. It is important to notice that in the case $F=\C$ the squared edge length~$l_{uv}(P)$ is not the square of the standard \textit{Hermitian\/} distance between the points~$P(u)$ and~$P(v)$. Instead, $l_{uv}(P)$ is defined by means of the standard \textit{bilinear\/} scalar product in~$\C^n$. So it is possible, for instance, that $l_{uv}(P)=0$, but $P(u)\ne P(v)$. 

The author is grateful to I.\,Kh.~Sabitov for attracting the author's attention to the problem of computing the volume of a polyhedron from the edge lengths and for multiple fruitful discussions. The author wishes to thank V.\,M.~Buchstaber, I.\,A.~Dynnikov, S.\,A.~Gaifullin, S.\,O.~Gorchinskiy, and S.\,Yu.~Rybakov for useful comments.

\section{Cayley--Menger determinants}

Let $p_0,p_1,\ldots,p_n$ be points in a Euclidean space and let $l_{ij}$ denote the square of the distance between~$p_i$ and~$p_j$. The \textit{Cayley--Menger determinant\/} of the points $p_0,p_1,\ldots,p_n$ is the $(n+2)\times(n+2)$-determinant given by
\begin{equation*}
CM(p_0,\ldots,p_n)=\left|
\begin{matrix}
0 & 1 & 1 & 1 & \cdots & 1\\
1 & 0 & l_{01} & l_{02} & \cdots & l_{0n}\\
1 & l_{01} & 0 & l_{12} & \cdots & l_{1n}\\
1 & l_{02} & l_{12} & 0 & \cdots & l_{2n}\\
\vdots & \vdots & \vdots & \vdots & \ddots & \vdots\\
1 & l_{0n} & l_{1n} & l_{2n} & \cdots & 0
\end{matrix}
\right|
\end{equation*} 
This determinant was introduced by A.~Cayley~\cite{Cay41} for $n=4$ and by K.~Menger~\cite{Men28}, \cite{Men31} for an arbitrary $n$ (see also~\cite{Blu70}). Their result is the following. \textit{The points $p_0,\ldots,p_n$ are affinely dependent, that is, are contained in an $(n-1)$-dimensional plane if and only if $CM(p_0,\ldots,p_n)=0$.
Moreover, if $T$ is an $n$-dimensional simplex  with vertices $p_0,\ldots,p_n$, then the square of its $n$-dimensional volume is given by}
\begin{equation}\label{eq_CMvol}
V^2=\frac{(-1)^{n+1}}{2^n(n!)^2}CM(p_0,\ldots,p_n).
\end{equation}

Since a polynomial with real coefficients is uniquely determined by the polynomial function associated to it, we see that the same holds for universal points. We mean the following. 

First, let $p_0,\ldots,p_n$ be universal $(n-1)$-dimensional points, that is, $(n-1)$-tuples of independent variables $(x_{i,1},\ldots,x_{i,n-1})$, $i=0,\ldots,n$, and let $l_{ij}$ be the universal squared distance between $p_i$ and~$p_j$. Then the equality $CM(p_0,\ldots,p_n)=0$ holds in the polynomial ring $\Q[x_{0,1},\ldots,x_{n,n-1}]$ in variables~$x_{i,k}$, $i=0,\ldots,n$, $k=1,\ldots,n-1$.

Second, consider the cycle $\partial\eta$ for some $n$-dimensional oriented simplex~$\eta\subset\Delta^M$. Then the universal volume~$V=V_{\partial\eta}$ of the simplex~$\eta$ satisfies~\eqref{eq_CMvol}, where $p_0,\ldots,p_n$ are the universal vertices of~$\eta$.

\section{Scheme of the proof}

Let $K$ be a finite simplicial complex.  A \textit{degree} of a vertex $u\in K$ is the number of edges $e\in K$ containing~$u$.

All known proofs of Theorem~\ref{theorem_main3} are based on the induction on the genus of the surface~$K$, then, for a given genus, on the number of vertices of~$K$, and finally, for given genus and number of vertices, on the smallest vertex degree of~$K$. The scheme of the induction is the following. If $K$ contains a vertex of degree~$3$, then one can delete this vertex and replace its star by a single triangle. Otherwise, one can decrease the degree of the vertex with the smallest degree by performing a flip, see Fig.~\ref{fig_flip}. The only obstruction to such operations is a \textit{splitting triangle} for $K$ consisting of three edges $[ij]$, $[jk]$, and $[ki]$ of~$K$ such that $[ijk]$ is not a triangle of~$K$. If $K$ contains a splitting triangle, then one can cut $K$ along this triangle and obtain either two polyhedra each with smaller number of vertices or a single polyhedron of smaller genus. The inductive step is the following lemma, which is referred later as Sabitov's Key Lemma. We formulate this lemma not to explain the Sabitov proof, but only to clarify the scheme of the induction.
\begin{figure}
\unitlength=3mm
\begin{center}
\begin{picture}(23,6)

\multiput(0,0)(15,0){2}{%
\begin{picture}(0,0)
\put(0,3){\line(2,3){2}}
\put(0,3){\line(2,-3){2}}
\put(2,0){\line(2,3){4}}
\put(2,6){\line(2,-3){4}}
\put(6,6){\line(2,-3){2}}
\put(6,0){\line(2,3){2}}
\put(2,0){\line(1,0){4}}
\put(2,6){\line(1,0){4}}
\put(4,3){\circle*{.3}}
\end{picture}}

\put(0,3){\line(1,0){8}}
\put(10,3){\vector(1,0){3}}
\put(15,3){\line(1,0){4}}
\put(21,0){\line(0,1){6}}

\end{picture}
\end{center}
\caption{Decreasing the degree of a vertex  by a flip}
\label{fig_flip}
\end{figure}
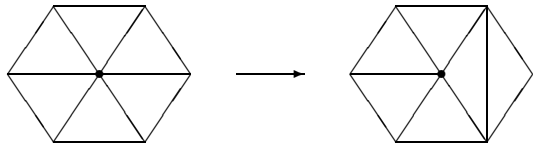

\begin{SKL}
Assume that the assertion of Theorem~\ref{theorem_main3} holds for any triangulated oriented surface that satisfies one of the following conditions:
\begin{itemize}
\item has genus less than~$g$, 
\item has genus~$g$ and has less than~$m$ vertices,
\item has genus~$g$, has exactly $m$ vertices, and has the smallest vertex degree less than~$p$.
\end{itemize}
Let  $K$ be a triangulated oriented surface of genus~$g$ that has $m$ vertices with the smallest vertex degree equal to~$p$. Then the assertion of Theorem~\ref{theorem_main3} holds for~$K$.
\end{SKL}

In dimensions greater than~$3$ we have two difficulties:

1. It seems to be difficult to find operations decreasing the number of vertices of a polyhedron and preserving the topological type of it.

2. It seems to be even more difficult to find some complexity function (like genus in dimension~$3$) of topological types of polyhedra that can be decreased by certain surgery.

The idea of our proof is to neglect the topological type of the support of a cycle and to make induction only on certain quantitative characteristics of the supports of cycles such as the number of vertices and the degrees of vertices. It is also very important that we use cycle polyhedra rather than pseudomanifold polyhedra. This approach allows us to simplify a cycle~$Z\in\mathcal{C}_3(\Delta^M)$ by means of the  \textit{elementary moves} that transform the cycle~$Z$ to cycles $Z-\partial\eta$, where $\eta$ is an oriented $4$-face of~$\Delta^M$.
We shall mostly use the universal language described in the end of section~\ref{section_intro}. This means that we consider the universal squared edge lengths~$l_{uv}$ and the universal volume~$V_Z$ for an $(n-1)$-dimensional cycle~$Z$ and we do not deal with particular polyhedra $P:Z\to\R^n$.

Let $Z$ be a $3$-dimensional cycle and let $K$ be its support. 
Let $m$ be the number of vertices of~$K$ and let $p$ be the smallest degree of a vertex in~$K$. Any vertex~$u\in K$ of degree~$p$ will be called a \textit{minimal degree vertex\/} of~$K$. We shall prove Theorem~\ref{theorem_alg4} by induction on~$m$ and, for a fixed $m$, by induction on~$p$. For the base of the induction, we can take $Z=0$. Then  $K=\emptyset$, $m=0$, $R_K=\Q$, and $V_Z=0$, hence, Theorem~\ref{theorem_alg4} obviously holds. Now assume that Theorem~\ref{theorem_alg4} holds for any $3$-cycle whose support either has less than~$m$ vertices or has exactly $m$ vertices and has the smallest vertex degree less than~$p$. The proof of the induction step will be divided into two lemmas.

\begin{lem}\label{lem_step}
Assume that the assertion of Theorem~\ref{theorem_alg4} holds for any $3$-cycle whose support either has less than~$m$ vertices or has exactly $m$ vertices and has the smallest vertex degree less than~$p$. Let $Z$ be a $3$-cycle whose support $K$ has $m$ vertices with the smallest vertex degree equal to~$p$. Let $u$ be a minimal degree vertex of~$K$, that is, a vertex of degree~$p$ and let $v\in K$ be an arbitrary vertex joined by an edge with~$u$ in~$K$. Then, for some $s\ge 0$, the element~$l_{uv}^sV_Z$ is integral over~$R_K$. 
\end{lem}

\begin{lem}\label{lem_alg}
Let $Z$ be a $3$-cycle with support~$K$. Let $u$ be a vertex of~$K$ and let $v_1,\ldots,v_p$ be all vertices joined by edges with~$u$ in~$K$. Assume that, for every~$i$, there exists a nonnegative integer~$s_i$ such that the element~$l_{uv_i}^{s_i}V_Z$ is integral over~$R_K$. Then the element~$V_Z$ is also integral over~$R_K$. 
\end{lem}

Obviously, these two lemmas imply the assertion of Theorem~\ref{theorem_alg4} for any cycle~$Z$ whose support has $m$ vertex and has the smallest vertex degree equal to~$p$, which completes the inductive proof of Theorem~\ref{theorem_alg4}.

Lemma~\ref{lem_step} is a $4$-dimensional analogue of Sabitov's Key Lemma. There exist two different proofs of Sabitov's Key Lemma. The first one is constructive; it is due to I.\,Kh.~Sabitov~\cite{Sab96}--\cite{Sab98b}. The second one was obtained by R.~Connelly, I.\,Kh.~Sabitov, and A.~Walz~\cite{CSW97}. It is simpler than Sabitov's proof, though non-constructive and less elementary. We shall also give two proofs of Lemma~\ref{lem_step}. The first one will be in spirit of Connelly--Sabitov--Walz; it will be given in section~\ref{section_nonconstr}. The second proof given in section~\ref{section_constr} will be constructive; it will be similar to the initial Sabitov proof of his Key Lemma. 

Nevertheless, there are two important differences between Sabitov's Key Lemma and Lemma~\ref{lem_step}. The first is that we consider a more general definition of a polyhedron, which allows us to use a simpler induction. This simplification is crucial because it is not known how to generalize directly Sabitov's induction procedure to higher dimensions. The second difference is the appearance of the multipliers~$l_{uv}^s$. This is a new phenomenon, which does not exist in dimension~$3$. In section~\ref{section_high} we shall give an analogue of Lemma~\ref{lem_step} for all dimensions~$n\ge 3$. This general lemma will clarify the cause for the absence of the multiplier~$l_{uv}^s$ in dimension~$3$.
The multiplier~$l_{uv}^s$ makes Lemma~\ref{lem_alg} necessary for the proof of Theorem~\ref{theorem_alg4}. In dimension~$3$ we do not need any analogue of this lemma. The proof of Lemma~\ref{lem_alg} will be given in section~\ref{section_alg}. Unfortunately, our proof of Lemma~\ref{lem_alg} is not constructive. Hence the constructive proof of Lemma~\ref{lem_step} given in section~\ref{section_constr} still does not allow us to obtain a constructive proof of Theorem~\ref{theorem_main4}.
Notice that the only obstruction to spreading our proof to dimensions $n\ge 5$ is that we cannot prove a proper multidimensional analogue of Lemma~\ref{lem_alg}. All other parts of our proof can be easily generalized to any dimension $n\ge 5$ (see section~\ref{section_high}).

\section{Non-constructive proof of Lemma~\ref{lem_step}}\label{section_nonconstr}

The main tool of the proof is theory of places. Recall some basic facts on places, see~\cite{Lang} for details. Let $F$ be a field. We add to~$F$ an element~$\infty$ and, for every~$a\in F$, put $a\pm\infty=\infty\cdot\infty=\frac10=\infty$, $\frac{a}{\infty}=0$, and if $a\ne 0$, $a\cdot\infty=\infty$. The expressions~$\frac00$, $\frac{\infty}{\infty}$, $0\cdot\infty$, $\infty\pm\infty$ are not defined.

A \textit{place\/} of a field~$L$ into a field~$F$ is a mapping~$\varphi:L\to F\cup\{\infty\}$ such that 
$\varphi(x+y)=\varphi(x)+\varphi(y)$ and $\varphi(xy)=\varphi(x)\varphi(y)$ whenever the right-hand sides are defined, and $\varphi(1)=1$. Elements of~$F$ are called \textit{finite\/} and we say that a place~$\varphi$ is finite on an element $x\in L$ if $\varphi(x)\in F$. We shall use the following standard lemma (see~\cite[page~12]{Lang}).

\begin{lem}
An element $x$ in a field~$L$ containing the ring~$R$ is integral over~$R$ if and only if every place defined on~$L$ that is finite on~$R$ is finite on~$x$.
\end{lem}

Recall that a \textit{universal $n$-dimensional point\/} is an $n$-tuple 
$(x_1,\ldots,x_n)$ of independent variables, which are called the \textit{coordinates\/} of the point. The coordinates of a universal point~$w$ will be denoted by $x_{w,1},\ldots,x_{w,n}$. If $w_1,\ldots,w_r$ are universal points, we denote by~$L(w_1,\ldots,w_r)$ the field of rational functions in $nr$ independent variables~$x_{w_j,i}$ with coefficients in~$\Q$.  

\begin{lem}\label{lem_place1}
Let $u$, $v$, $w_1$, $w_2$, $w_3$, $w_4$ be universal $4$-dimensional points. Let $\varphi$ be a place of the field $L(u,v,w_1,w_2,w_3,w_4)$ into a field $F$ of characteristic not equal to~$2$ such that $\varphi(l_{uv})$ is neither~$0$ nor~$\infty$,
$$
\varphi(l_{w_1w_3})=\varphi(l_{w_2w_4})=
\varphi\left(\frac{l_{w_1w_3}}{l_{w_1w_2}}\right)=
\varphi\left(\frac{l_{w_2w_4}}{l_{w_3w_4}}\right)=\infty,
$$
and $\varphi(l_{uw_j})$, $\varphi(l_{vw_j})$, $j=1,2,3,4$, and  $\varphi(l_{w_2w_3})$ are finite. Then
$$
\varphi\left(\frac{l_{w_1w_4}}{l_{w_1w_3}}\right)=
\varphi\left(\frac{l_{w_1w_4}}{l_{w_2w_4}}\right)=\infty.
$$ 
\end{lem}

\begin{proof}
The Cayley--Menger determinant of six $4$-dimensional points is equal to zero:
\begin{eqnarray*}
\left|
\begin{array}{ccccccc}
0 & 1 & 1 & 1 & 1 & 1 & 1 \\
1 & 0 & l_{uv} & l_{uw_1} & l_{uw_2} & l_{uw_3} & l_{uw_4} \\
1 & l_{uv} & 0 & l_{vw_1} & l_{vw_2} & l_{vw_3} & l_{vw_4} \\
1 & l_{uw_1} & l_{vw_1} & 0 & l_{w_1w_2} & l_{w_1w_3} & l_{w_1w_4} \\
1 & l_{uw_2} & l_{vw_2} & l_{w_1w_2} & 0 & l_{w_2w_3} & l_{w_2w_4} \\
1 & l_{uw_3} & l_{vw_3} & l_{w_1w_3} & l_{w_2w_3} & 0 & l_{w_3w_4} \\
1 & l_{uw_4} & l_{vw_4} & l_{w_1w_4} & l_{w_2w_4}  & l_{w_3w_4} & 0
\end{array}
\right|
=0.
\end{eqnarray*}
Dividing the $4$th row and the $4$th column by~$l_{w_1w_3}$ and dividing the $7$th row and the $7$th column by~$l_{w_2w_4}$, we obtain 
\begin{eqnarray}\label{eq_det7}
\left|
\begin{array}{ccccccc}
0 & 1 & 1 & \displaystyle\frac{1}{l_{w_1w_3}}\vphantom{\frac{\bigl(}{\bigl(}} & 1 & 1 & \displaystyle\frac{1}{l_{w_2w_4}} \\
1 & 0 & l_{uv} & \displaystyle\frac{l_{uw_1}}{l_{w_1w_3}}\vphantom{\frac{\bigl(}{\bigl(}} & l_{uw_2} & l_{uw_3} & \displaystyle\frac{l_{uw_4}}{l_{w_2w_4}} \\
1 & l_{uv} & 0 & \displaystyle\frac{l_{vw_1}}{l_{w_1w_3}} & l_{vw_2} & l_{vw_3} & \displaystyle\frac{l_{vw_4}}{l_{w_2w_4}} \\
\displaystyle\frac{1}{l_{w_1w_3}} & \displaystyle\frac{l_{uw_1}}{l_{w_1w_3}} & \displaystyle\frac{l_{vw_1}}{l_{w_1w_3}} & 0\vphantom{\frac{\bigl(}{\bigl(}} & \displaystyle\frac{l_{w_1w_2}}{l_{w_1w_3}} & 1 & \displaystyle\frac{l_{w_1w_4}}{l_{w_1w_3}l_{w_2w_4}} \\
1 & l_{uw_2} & l_{vw_2} & \displaystyle\frac{l_{w_1w_2}}{l_{w_1w_3}}\vphantom{\frac{\bigl(}{\bigl(}} & 0 & l_{w_2w_3} & 1 \\
1 & l_{uw_3} & l_{vw_3} & 1\vphantom{\frac{\bigl(}{\bigl(}} & l_{w_2w_3} & 0 & \displaystyle\frac{l_{w_3w_4}}{l_{w_2w_4}} \\
\displaystyle\frac{1}{l_{w_2w_4}} & \displaystyle\frac{l_{uw_4}}{l_{w_2w_4}} & \displaystyle\frac{l_{vw_4}}{l_{w_2w_4}} & \displaystyle\frac{l_{w_1w_4}}{l_{w_1w_3}l_{w_2w_4}}\vphantom{\frac{\bigl(}{\bigl(}} & 1  & \displaystyle\frac{l_{w_3w_4}}{l_{w_2w_4}} & 0
\end{array}
\right|
=0.
\end{eqnarray}
Suppose, $\varphi\left(\frac{l_{w_1w_4}}{l_{w_1w_3}l_{w_2w_4}}\right)=0$. Then all entries of the matrix in the left-hand side of~\eqref{eq_det7} would have finite $\varphi$ values and all entries in the $4$th and in the $7$th rows and in the $4$th and in the $7$th columns would have zero $\varphi$ values except the four entries equal to~$1$. Expanding the determinant in the $4$th and the $7$th rows and in the $4$th and the $7$th columns we obtain that the determinant
\begin{eqnarray*}
\left|
\begin{array}{ccc}
0 & 1 & 1 \\
1 & 0 & \varphi(l_{uv}) \\
1 & \varphi(l_{uv}) & 0
\end{array}
\right|=2\varphi(l_{uv})
\end{eqnarray*}
is equal to~$0$. Since the characteristic of the field~$F$ is not~$2$, this contradicts the condition~$\varphi(l_{uv})\ne 0$. Hence, $\varphi\left(\frac{l_{w_1w_4}}{l_{w_1w_3}l_{w_2w_4}}\right)\ne 0$. But $\varphi(l_{w_1w_3})=\varphi(l_{w_2w_4})=\infty$. Therefore, $\varphi\left(\frac{l_{w_1w_4}}{l_{w_1w_3}}\right)=\varphi\left(\frac{l_{w_1w_4}}{l_{w_2w_4}}\right)=\infty$.
\end{proof}

\begin{lem}\label{lem_place2}
Let $u$, $v$, $w_1, \ldots,w_r$ be universal $4$-dimensional points, $r\ge 4$. Let $\varphi$ be a place of the field $L(u,v,w_1,\ldots,w_r)$ into a field $F$ of characteristic not equal to~$2$ such that $\varphi(l_{uv})$ is neither~$0$ nor~$\infty$ and
 $\varphi(l_{uw_j})$, $\varphi(l_{vw_j})$, $\varphi(l_{w_jw_{j+1}})$ are finite for $j=1,\ldots,r$ $(j+1\mod r)$. Then for some $j$,
$
\varphi(l_{w_{j}w_{j+2}})$ is finite, $j=1,2,\ldots,r-2$. 
\end{lem}

\begin{proof}
Suppose $\varphi(l_{w_jw_{j+2}})=\infty$ for $j=1,2,\ldots,r-2$. We shall prove inductively for $j=3,\ldots,r$, that $\varphi(l_{w_1w_j})=\varphi\left(\frac{l_{w_1w_j}}{l_{w_1w_{j-1}}}\right)=\infty$. This is clear for $j=3$. Now assume it for $j$ and prove it for~$j+1$. The $6$ points $u$, $v$, $w_1$, $w_{j-1}$, $w_j$, $w_{j+1}$ satisfy the conditions of Lemma~\ref{lem_place1}. The conclusion of Lemma~\ref{lem_place1} yields~$\varphi\left(\frac{l_{w_1w_{j+1}}}{l_{w_1w_j}}\right)=\infty$. Since $\varphi(l_{w_1w_j})=\infty$, we also obtain that~$\varphi(l_{w_1w_{j+1}})=\infty$. This completes the induction step. For $j=r$, we obtain $\varphi(l_{w_1w_r})=\infty$, which contradicts the assumptions of Lemma~\ref{lem_place2}. This contradiction yields that some~$\varphi(l_{w_jw_{j+2}})$ is finite.
\end{proof}

\begin{remark}
Lemmas~\ref{lem_place1} and~\ref{lem_place2} are $4$-dimensional analogues of Lemmas~3 and~4 in~\cite{CSW97}. However, compairing to the $3$-dimensional case we have to add  an essentially new condition~$\varphi(l_{uv})\ne 0$. This condition will lead us to the multiplier~$l_{uv}^s$ in Lemma~\ref{lem_step}.
\end{remark}

Now let $Z$ be a nonzero $3$-cycle and let $K$ be its support. 
For an oriented edge~$[uv]$ of~$K$, we define a $1$-cycle~$Z_{[uv]}$, which will be called the \textit{link\/} of~$[uv]$ in~$Z$. Suppose, 
\begin{equation}\label{eq_Z}
Z=\sum_{i=1}^kq_i[uvw_it_i]+A,
\end{equation}
where $q_i\in\Z$ and $A$ is a sum of $3$-simplices that do not contain the edge~$[uv]$. 
Then we put 
\begin{equation*}
Z_{[uv]}=\sum_{i=1}^kq_i[w_it_i].
\end{equation*}
Obviously, the chain~$Z_{[uv]}$ is well defined. This means that it is independent of the choice of decomposition of~$Z$ into a sum of the form~\eqref{eq_Z}. If $Z_{[uv]}$ were~$0$, the edge~$[uv]$ would not belong to the support of~$Z$. Hence, $Z_{[uv]}\ne 0$.
We have 
$$
0=\partial Z=\sum_{i=1}^kq_i([uvt_i]-[uvw_i])+B,
$$
where $B$ is a sum of $2$-simplices that do not contain the edge~$[uv]$. Hence,
$$
\sum_{i=1}^kq_i([uvt_i]-[uvw_i])=0.
$$
Therefore,
$$
\partial Z_{[uv]}=\sum_{i=1}^kq_i([t_i]-[w_i])=0.
$$
Thus $Z_{[uv]}$ is a $1$-cycle.

Now it is convenient to assume that the simplices~$[uvw_it_i]$ are pairwise distinct. We mean that for any $i\ne j$ neither $[w_it_i]=[w_jt_j]$ nor  $[w_it_i]=[t_jw_j]$. A decomposition~\eqref{eq_Z} with this property is unique up to permuting the addends and replacing each addend~$q_i[uvw_it_i]$ by $-q_i[uvt_iw_i]$. Consider a well-defined positive number $q(Z_{[uv]})=\sum_{i=1}^k|q_i|$.  Equivalently, $q(Z_{[uv]})$ is the sum of absolute values of the coefficients in~$Z$ of all simplices containing the edge~$[uv]$. Since $Z_{[uv]}$ is a nonzero $1$-cycle, we see that  $q(Z_{[uv]})\ge 3$. 

Let $C$ be a $1$-chain. We shall assign to~$C$ a weighted oriented graph~$\Gamma_C$. This graph consists of all oriented edges~$[w_1w_2]$ that enter the chain~$C$ with positive coefficients, which are called the \textit{weights\/} of the corresponding edges. Obviously, $C$ is a cycle if and only if, for every vertex $w$, the sum of weights
of incoming edges is equal to the sum of weights of outgoing edges. This construction motivates the following definition.
A sequence of pairwise distinct vertices $w_1,\ldots,w_r$ is called a \textit{directed simple cycle\/} in~$C$ if the edges $[w_1w_2]$, $[w_2w_3],\ldots,[w_{r-1}w_r]$, and $[w_rw_1]$ enter the $1$-chain~$C$ with positive coefficients. The following lemma is straightforward.

\begin{lem}\label{lem_dircycl}
Any $1$-cycle contains a directed simple cycle.
\end{lem}

\begin{proof}[Proof of Lemma~\ref{lem_step}]
The proof is by induction on~$q(Z_{[uv]})$. To get the base of the induction we first consider the case~$q(Z_{[uv]})=3$. Then 
$$
Z_{[uv]}=[w_1w_2]+[w_2w_3]+[w_3w_1]
$$
for some pairwise distinct vertices~$w_1$, $w_2$, $w_3$. Hence,
$$
Z=[uvw_1w_2]+[uvw_2w_3]+[uvw_3w_1]+A,
$$
where $A$ is a sum of $3$-simplices that do not contain the edge~$[uv]$. Consider the $4$-simplex $\eta=[uvw_1w_2w_3]$ and the $3$-cycle
$$
Z'=Z-\partial\eta=[uw_1w_2w_3]-[vw_1w_2w_3]+A.
$$
Let $K'$ be the support of~$Z'$. Every vertex of~$\eta$ is a vertex of~$K$ and every edge of~$\eta$ is an edge of~$K$. Hence every vertex of~$K'$ is a vertex of~$K$ and every edge of~$K'$ is an edge of~$K$. On the other hand, we see that the cycle~$Z$ is a sum of simplices that do not contain the edge~$[uv]$, hence, $[uv]\notin K'$. Therefore, either $u\notin K'$ and then the number of vertices of~$K'$ is less than~$m$, or $u\in K'$ and then the the number of vertices of~$K'$ is less or equal to~$m$ and the degree of~$u$ is less than~$p$. Consequently, we know that the assertion of Theorem~\ref{theorem_alg4} holds for the cycle~$Z'$, that is, the element~$V_{Z'}$ is integral over the ring~$R_{K'}$. The element~$V_{\partial\eta}$ is the universal volume of the simplex~$\eta$. The Cayley--Menger formula~\eqref{eq_CMvol} implies that $V_{\partial\eta}$ is integral over the $\Q$-algebra~$R_{\eta}$ generated by the squares of the universal lengths of edges of~$\eta$. Since all edges of~$K'$ and all edges of~$\eta$ belong to~$K$, we have $R_{K'}\subset R_K$ and $R_{\eta}\subset R_K$. Therefore, the elements~$V_{Z'}$ and~$V_{\partial\eta}$ are integral over~$R_K$, hence, the element~$V_Z=V_{Z'}+V_{\partial\eta}$ is also integral over~$R_K$. This is the base of the induction.

Now we assume that Lemma~\ref{lem_step} is proved for $q(Z_{[uv]})<q$ and we
prove it for $q(Z_{[uv]})=q$. By Lemma~\ref{lem_dircycl} there is a directed simple cycle $w_1,\ldots,w_r$ in~$Z_{[uv]}$. 
For $j=1,\ldots,r-2$, we put $\eta_j=[uvw_jw_{j+1}w_{j+2}]$, $Z_j=Z-\partial\eta_j$, and $K_j=\supp(Z_j)$. Every vertex of~$\eta_j$ is a vertex of~$K$. Hence, every vertex of~$K_j$ is a vertex of~$K$. Therefore, the number of vertices of~$K_j$ is not greater than~$m$. Besides, the edges $[uv]$, $[uw_j]$, $[uw_{j+1}]$, and~$[uw_{j+2}]$ belong to~$K$. Hence, every edge that belongs to~$K_j$ and contains~$u$ belongs to~$K$. Therefore, the degree of the vertex~$u$ in~$K_j$ is not greater than~$p$. We consider two cases:

1. Either the number of vertices of~$K_j$ is smaller than~$m$, or the number of vertices of~$K_j$ is equal to~$m$ and the smallest vertex degree in~$K_j$ is less than~$p$. Then by the assumption of Lemma~\ref{lem_step}, Theorem~\ref{theorem_alg4} holds for~$Z_j$. Therefore, the element~$V_{Z_j}$ is integral over~$R_{K_j}$.

2. The number of vertices of~$K_j$ is equal to~$m$ and the smallest vertex degree in~$K_j$ is equal to~$p$. Then $u$ is a minimal degree vertex of~$K_j$. We have
$$
(Z_j)_{[uv]}=Z_{[uv]}-[w_jw_{j+1}]-[w_{j+1}w_{j+2}]+[w_jw_{j+2}].
$$
The edges~$[w_jw_{j+1}]$ and~$[w_{j+1}w_{j+2}]$ enter the $1$-cycle~$Z_{[uv]}$ with positive coefficients. Hence the absolute value of either of their coefficients in~$(Z_j)_{[uv]}$ becomes $1$ smaller than in~$Z_{[uv]}$. The absolute value of the coefficient of the edge~$[w_jw_{j+2}]$ in~$(Z_j)_{[uv]}$ becomes either $1$ smaller than in~$Z_{[uv]}$, or $1$ larger than in~$Z_{[uv]}$. Therefore, $q\bigl((Z_j)_{[uv]}\bigr)$ is either $q(Z_{[uv]})-1$ or $q(Z_{[uv]})-3$, hence, it is smaller than~$q(Z_{[uv]})$. By the induction assumption, for some~$s_j\ge 0$, the element~$l_{uv}^{s_j}V_{Z_j}$ is integral over~$R_{K_j}$.

We see that in both cases for some~$s_j\ge 0$, the element~$l_{uv}^{s_j}V_{Z_j}$ is integral over~$R_{K_j}$. 
 All edges of~$K_j$ are edges of~$K$ possibly except the edge~$[w_{j}w_{j+2}]$. Hence $R_{K_j}\subset R_K[l_{w_jw_{j+2}}]$. Similarly, all edges of the simplex~$\eta_j$ are edges of~$K$ possibly except the edge~$[w_{j}w_{j+2}]$. Hence $R_{\eta_j}\subset R_K[l_{w_jw_{j+2}}]$ Therefore, the elements~$l_{uv}^{s_j}V_{Z_j}$ and $V_{\partial\eta_{j}}$ are integral over~$R_K[l_{w_jw_{j+2}}]$. We have $V_Z=V_{Z_j}+V_{\partial\eta_{j}}$, hence, the element $l_{uv}^{s_j}V_Z$ is integral over~$R_K[l_{w_jw_{j+2}}]$. Consequently, the element~$V_Z$ is integral over the ring~$R_K\left[l_{w_jw_{j+2}},\frac{1}{l_{uv}}\right]$.
If $r=3$, then $[w_jw_{j+2}]$ is an edge of~$K$ and the assertion of Lemma~\ref{lem_step} follows.

Suppose, $r\ge 4$. We need to prove that, for some~$s\ge 0$, the element~$l_{uv}^sV_Z$ is integral over~$R_K$. Equivalently, we need to prove that the element~$V_Z$ is integral over the ring~$R_K\bigl[\frac{1}{l_{uv}}\bigr]$. We shall prove that every place~$\varphi:\Q(x_K)\to F\cup\{\infty\}$ that is finite on~$R_K\bigl[\frac{1}{l_{uv}}\bigr]$ is finite on~$V_Z$. Since the place~$\varphi$ is finite on~$\Q\subset R_K$, we see that the field~$F$ has zero characteristic. The edges~$[uw_j]$, $[vw_j]$, $[w_jw_{j+1}]$, $j=1,\ldots,r$ $(j+1 \mod r)$, and $[uv]$ belong to~$K$. Hence the place~$\varphi$ is finite on their squares of lengths. Since~$\varphi\bigl(\frac{1}{l_{uv}}\bigr)\ne\infty$, we also have $\varphi(l_{uv})\ne 0$. Therefore, we can apply Lemma~\ref{lem_place2} to the points~$u$, $v$, $w_1,\ldots,w_r$. We obtain that for some~$j$, $\varphi(l_{w_jw_{j+2}})$ is finite, $j=1,\ldots,r-2$. Then the place~$\varphi$ is finite on the ring~$R_K\left[l_{w_jw_{j+2}},\frac{1}{l_{uv}}\right]$. Since the element~$V_Z$ is integral over this ring, we obtain that the place~$\varphi$ is finite on~$V_Z$. This completes the proof of the induction step.
\end{proof}

\section{Constructive proof of  Lemma~\ref{lem_step}}\label{section_constr}

The element~$l_{uv}^sV_Z$ is integral over~$R_K$ if and only if there exists a polynomial relation
of the form
\begin{equation}\label{eq_Sabuv}
Q(V_Z,l_K)=l_{uv}^sV^N_Z+a_1(l_K)V^{N-1}_Z+\ldots + a_N(l_K)=0,
\end{equation}
where $l_K$ denotes the set of all variables $l_{ww'}=l_{w'w}$ corresponding to edges $[ww']\in K$ and $a_j$ are polynomials in the variables~$l_{ww'}$ with rational coefficients.
In this section we shall give a constructive proof of Lemma~\ref{lem_step}. This means that
we shall explicitly reduce the problem of finding the polynomial relation~\eqref{eq_Sabuv} to the problem of finding Sabitov polynomials~\eqref{eq_Sab} for cycles whose supports either have less than $m$ vertices or have $m$ vertices with the smallest vertex degree less than~$p$.
As in the previous section, we proceed by induction on~$q(Z_{[uv]})$.
 Let $w_1,w_2,\ldots,w_r$ be a directed simple cycle in~$Z_{[uv]}$. 
For $j=1,\ldots,r-2$, we put $\eta_j=[uvw_jw_{j+1}w_{j+2}]$, $Z_j=Z-\partial\eta_j$, and $K_j=\supp(Z_j)$. In the previous section we have shown that for every~$j$, one of the following three statements holds:
\begin{itemize}
\item $K_j$ has less than~$m$ vertices; \item
 $K_j$ has $m$ vertices and the smallest  degree of a vertex in~$K_j$ is less than~$p$; \item $K_j$ has $m$ vertices, the smallest degree of a vertex in~$K_j$ is equal to~$p$, and $q\left((Z_j)_{[uv]}\right)<q(Z_{[uv]})$.   
\end{itemize}
In the first two cases the assumption of Lemma~\ref{lem_step} yields that we have a Sabitov polynomial for~$Z_j$, that is, a polynomial relation~$Q_j(V_{Z_j},l_{K_j})=0$ that is monic with respect to~$V_{Z_j}$. In the third case the inductive assumption yields that we have a polynomial relation~$Q_j(V_{Z_j},l_{K_j})=0$ with the coefficient of the largest power of~$V_{Z_j}$ equal to~$l_{uv}^{s_j}$ for some $s_j\ge 0$. Putting~$s_j=0$ in the first two cases, we obtain that in all cases we have the polynomial relation of the form
\begin{equation}\label{eq_Zj(1)}
l_{uv}^{s_j}V_{Z_j}^{N_j}+a_{j,1}(l_{K_j})V_{Z_j}^{N_j-1}+\cdots+a_{j,N_j}(l_{K_j})=0,
\end{equation}
where $a_{j,i}$ are polynomials in~$l_{K_j}$ with rational coefficients.
Following~\cite{Sab96} we introduce notation~$d_j=l_{w_jw_{j+2}}$ and $D_j=l_{w_1w_j}$. All edges of~$K_j$ are edges of~$K$ except possibly the edge~$[w_jw_{j+2}]$. Hence relation~\eqref{eq_Zj(1)} can be rewritten in the form
\begin{equation}\label{eq_Zj(2)}
l_{uv}^{s_j}V_{Z_j}^{N_j}+a_{j,1}(l_{K},d_j)V_{Z_j}^{N_j-1}+\cdots+a_{j,N_j}(l_{K},d_j)=0.
\end{equation}
We have $V_Z=V_{Z_j}+V_{\partial\eta_j}$. Hence we can rewrite~\eqref{eq_Zj(2)} in the form
\begin{equation}\label{eq_Zj(3)}
l_{uv}^{s_j}V_{Z}^{N_j}+\tilde a_{j,1}(l_{K},d_j,V_{\partial\eta_j})V_{Z}^{N_j-1}+\cdots+\tilde a_{j,N_j}(l_{K},d_j,V_{\partial\eta_j})=0,
\end{equation}
where $\tilde a_{j,i}$ are polynomials in~$l_K$, $d_j$, and~$V_{\partial\eta_j}$ with rational coefficients. 
On the other hand, the Cayley--Menger formula~\eqref{eq_CMvol} yields
\begin{eqnarray}\label{eq_CM_etaj}
V_{\partial\eta_j}^2+\frac{1}{9216}\left|
\begin{array}{cccccc}
0 & 1 & 1 & 1 & 1 & 1\\
1 & 0 & l_{uv} & l_{uw_j} & l_{uw_{j+1}} & l_{uw_{j+2}}\\
1 & l_{uv} & 0  & l_{vw_j} & l_{vw_{j+1}} & l_{vw_{j+2}}\\
1 & l_{uw_j}  & l_{vw_j} & 0 & l_{w_jw_{j+1}} & d_j\\
1 & l_{uw_{j+1}}  & l_{vw_{j+1}} & l_{w_jw_{j+1}} & 0 & l_{w_{j+1}w_{j+2}}\\
1 & l_{uw_{j+2}}  & l_{vw_{j+2}} & d_j & l_{w_{j+1}w_{j+2}} & 0
\end{array}
\right|=0.
\end{eqnarray}
Exclude the variable~$V_{\partial\eta_j}$ from~\eqref{eq_Zj(3)} and~\eqref{eq_CM_etaj} by means of the resultant. This means that we take the resultant of the left-hand sides of~\eqref{eq_Zj(3)} and~\eqref{eq_CM_etaj} with respect to the variable~$V_{\partial\eta_j}$. It can be easily seen that we obtain a relation of the form
\begin{equation*}
\refstepcounter{equation}\label{eq_Z(j)}
l_{uv}^{2s_j}V_{Z}^{2N_j}+b_{j,1}(l_{K},d_j)V_{Z}^{2N_j-1}+\cdots + b_{j,2N_j}(l_{K},d_j)=0,\eqno{(\theequation_j)}
\end{equation*}
where $b_{j,i}$ are polynomials in~$l_K$ and~$d_j$ with rational coefficients.

Further we shall work with various polynomial relations among elements of the set~$l_K$ and the elements~$d_j$, $D_j$, and~$V_Z$. It is convenient to understand any such relation as a polynomial relation among the elements~$d_j$, $D_j$, and~$V_Z$ with coefficients being polynomials in~$l_K$. So speaking about a coefficient of some monomial, say~$d_2^2D_3^2V_Z^2$, we always understand that this coefficient is a polynomial in~$l_K$. To avoid confusion, further in this section we use notation~$l_{ww'}$ for the squares of lengths of edges of~$K$ only and use~$d_j$, $D_j$ for the squares of lengths of diagonals of~$K$.

Suppose, $r=3$. Then put $j=1$ and notice that $[w_1w_3]$ is an edge of~$K$, hence,  the variable $d_1=l_{w_1w_3}$ is contained in~$l_K$. Therefore, (\ref{eq_Z(j)}${}_1$) is the required polynomial relation. 

Suppose, $r=4$. Then we have two relations~(\ref{eq_Z(j)}${}_1$) and~(\ref{eq_Z(j)}${}_2$).
Let $k_1$ be the degree of~(\ref{eq_Z(j)}${}_1$) with respect to~$d_1$ and let $k_2$ be the degree of~(\ref{eq_Z(j)}${}_2$) with respect to~$d_2$. Besides, we have the Cayley--Menger relation for the six $4$-dimensional points~$u$, $v$, $w_1$, $w_2$, $w_3$, and~$w_4$:
\begin{eqnarray}\label{eq_CM_4d1d2}
\left|
\begin{array}{ccccccc}
0 & 1 & 1 & 1 & 1 & 1 & 1\\
1 & 0 & l_{uv} & l_{uw_1} & l_{uw_2} & l_{uw_3} & l_{uw_4}\\
1 & l_{uv} & 0  & l_{vw_1} & l_{vw_2} & l_{vw_3} & l_{vw_4}\\
1 & l_{uw_1}  & l_{vw_1} & 0 & l_{w_1w_2} & d_1 & l_{w_1w_4}\\
1 & l_{uw_2}  & l_{vw_2} & l_{w_1w_2} & 0 & l_{w_2w_3} & d_2\\
1 & l_{uw_3}  & l_{vw_3} & d_1 & l_{w_2w_3} & 0 & l_{w_3w_4}\\
1 & l_{uw_4}  & l_{vw_4}  & l_{w_1w_4} & d_2  & l_{w_3w_4} & 0
\end{array}
\right|=2l_{uv}d_1^2d_2^2+\ldots=0.
\end{eqnarray}
The degree of the Cayley--Menger polynomial~\eqref{eq_CM_4d1d2} with respect to either of the variables~$d_1$ and~$d_2$ is equal to~$2$ and the coefficient of the monomial~$d_1^2d_2^2$ is equal to~$2l_{uv}$. The resultant of~\eqref{eq_CM_4d1d2} and~(\ref{eq_Z(j)}${}_2$) with respect to the variable~$d_2$ is the $(k_2+2)\times (k_2+2)$-determinant
{\begin{small}
$$
\left|
\begin{array}{ccccccc}
2l_{uv}d_1^2+\ldots & *  & \cdots & 0 & 0 & 0\\ 
0& 2l_{uv}d_1^2+\ldots &  \cdots & 0 & 0 & 0\\
\vdots &  \vdots & \ddots &\vdots&\vdots&\vdots\\
0 & 0 &  \cdots & 2l_{uv}d_1^2+\ldots & * & *\\
* & * &  \cdots & * & l_{uv}^{2s_2}V_Z^{2N_2}+\ldots & 0\\
0 & * &  \cdots & * & * & l_{uv}^{2s_2}V_Z^{2N_2}+\ldots 
\end{array}
\right|
$$
\end{small}}\\
The dots in the first $k_2$ diagonal entries stand for terms whose degrees with respect to~$d_1$ is less than~$2$, the  dots in the last $2$ diagonal entries stand for terms whose degrees with respect to~$V_Z$ is less than~$2N_2$.
The variable~$V_Z$ does not enter the first $k_2$ rows. For all entries in the last two rows except the diagonal ones, their  degrees with respect to~$V_Z$ are strictly less than~$2N_2$.
On the other hand, the variable~$d_1$ does not enter the last two rows and for every entry in the first $k_2$ rows, its degree  with respect to~$d_1$ is not greater than~$2$. Therefore, we obtain the polynomial relation of the form 
\begin{equation}\label{eq_VZd_1}
c_0(l_K,d_1)V_Z^{4N_2}+c_1(l_K,d_1)V_Z^{4N_2-1}+\ldots+c_{4N_2}(l_K,d_1)=0,
\end{equation}
where~$c_i$ are the polynomials in~$l_K$ and~$d_1$ such that their degrees with respect to~$d_1$ are not greater than~$2k_2$ and
$$
c_0(l_K,d_1)=2^{k_2}l_{uv}^{4s_2+k_2}d_1^{2k_2}+\ldots,
$$
where dots stand for terms whose degrees with respect to~$d_1$ are smaller than~$2k_2$. Hence rewriting the left-hand side of~\eqref{eq_VZd_1} as a polynomial in~$d_1$, we shall obtain that its leading coefficient is $2^{k_2}l_{uv}^{4s_2+k_2}V_Z^{4N_2}$ plus terms of smaller degrees with respect to~$V_Z$. On the other hand, rewriting the left-hand side of~(\ref{eq_Z(j)}${}_1$) as a polynomial in~$d_1$, we shall obtain that the free term is the polynomial in~$V_Z$ with leading term~$l_{uv}^{2s_1}V_Z^{2N_1}$, while the coefficient of every positive power of~$d_1$ is a polynomial in~$V_Z$ of degree strictly less than~$2N_1$. Therefore, the resultant of the left-hand sides of~\eqref{eq_VZd_1} and~(\ref{eq_Z(j)}${}_1$) is the polynomial in~$V_Z$ with leading term $2^{k_1k_2}l_{uv}^{4k_2s_1+4k_1s_2+k_1k_2}V_Z^{4k_2N_1+4k_1N_2}$. Thus we obtain the required polynomial relation
$$
l_{uv}^{s}V_Z^{N}+a_1(l_K)V_Z^{N-1}+\ldots+a_{N}(l_K)=0,
$$
where $s=4k_2s_1+4k_1s_2+k_1k_2$ and $N=4k_2N_1+4k_1N_2$.

Now suppose, $r\ge 5$. For $j=2,\ldots,r-2$, consider the Cayley--Menger relation for the $6$ points~$u$, $v$, $w_1$, $w_j$, $w_{j+1}$, and $w_{j+2}$:
\begin{equation*}\refstepcounter{equation}\label{eq_CMj}
\left|
\begin{array}{ccccccc}
0 & 1 & 1 & 1 & 1 & 1 & 1\\
1 & 0 & l_{uv} & l_{uw_1} & l_{uw_j} & l_{uw_{j+1}} & l_{uw_{j+2}}\\
1 & l_{uv} & 0  & l_{vw_1} & l_{vw_j} & l_{vw_{j+1}} & l_{vw_{j+2}}\\
1 & l_{uw_1}  & l_{vw_1} & 0 & D_j & D_{j+1} & D_{j+2}\\
1 & l_{uw_j}  & l_{vw_j} & D_j & 0 & l_{w_jw_{j+1}} & d_j\\
1 & l_{uw_{j+1}}  & l_{vw_{j+1}} & D_{j+1} & l_{w_jw_{j+1}} & 0 & l_{w_{j+1}w_{j+2}}\\
1 & l_{uw_{j+2}}  & l_{vw_{j+2}}  & D_{j+2} & d_j  & l_{w_{j+1}w_{j+2}} & 0
\end{array}
\right|=0.
\eqno(\theequation_j)
\end{equation*}
We shall need the following properties of the left-hand sides of this relations: 
\begin{description}
\item[\normalfont(CM${}_j$-1)] The degree with respect to~$d_j$ is equal to~$2$;
\item[\normalfont(CM${}_j$-2)] The total degree with respect to~$D_{j}$, $D_{j+1}$, and~$D_{j+2}$ is equal to~$2$;
\item[\normalfont(CM${}_j$-3)] The coefficient of the monomial~$d_j^2D_{j+1}^2$ is equal to~$2l_{uv}$;
\item[\normalfont(CM${}_j$-4)] No monomial is divisible by~$d_j^2D_j$.
\end{description}
Using resultants we shall exclude successively the variables~$D_4,\ldots,D_{r-2}$ from equations~(\ref{eq_CMj}${}_3$)--(\ref{eq_CMj}${}_{r-2}$). First, we shall exclude~$D_4$ from equations~(\ref{eq_CMj}${}_3$) and~(\ref{eq_CMj}${}_4$); second, we shall exclude~$D_5$ from the obtained equation and equation~(\ref{eq_CMj}${}_5$), and so on. The equation obtained after excluding~$D_{j}$ will have the form
\begin{equation*}\refstepcounter{equation}\label{eq_Frel}
F_j(l_K,D_3,D_{j+1},D_{j+2},d_3,\ldots ,d_j)=0.\eqno{(\theequation_j)}
\end{equation*}
We start from equation~(\ref{eq_Frel}${}_3$)=(\ref{eq_CMj}${}_3$). The equation~(\ref{eq_Frel}${}_j$) is the resultant of the equations~(\ref{eq_Frel}${}_{j-1}$) and~(\ref{eq_CMj}${}_j$) with respect to the variable~$D_j$, $j=4,\ldots,r-2$. We need the following properties of the polynomials~$F_j$:
\begin{description}
\item[\normalfont($F_j$-1)] The degree of~$F_j$ with respect to each~$d_i$ is equal to~$2^{j-2}$, $i=3,\ldots,j$; 
\item[\normalfont($F_j$-2)] The total degree of~$F_j$ with respect to~$D_3$, $D_{j+1}$, and~$D_{j+2}$ is equal to~$2^{j-2}$;
\item[\normalfont($F_j$-3)] The coefficient of the monomial~$(d_3\ldots d_jD_{j+1})^{2^{j-2}}$ in~$F_j$ is equal to~$(2l_{uv})^{2^{j-3}(j-2)}$.
\item[\normalfont($F_j$-4)] No monomial of~$F_j$ is divisible by~$d_3^{2^{j-2}}D_3$.
\end{description}
We shall prove inductively these four properties of~$F_j$. 
For $j=3$, they coincide with properties (CM${}_3$-1)--(CM${}_3$-4).
Assume ($F_{j-1}$-1)--($F_{j-1}$-4) and prove \mbox{($F_{j}$-1)}--($F_{j}$-4), $j\ge 4$. The resultant of~(\ref{eq_Frel}${}_{j-1}$) and~(\ref{eq_CMj}${}_j$) with respect to the variable~$D_j$ is a $(2^{j-3}+2)\times (2^{j-3}+2)$-determinant
\begin{eqnarray}\label{eq_detAB}
\left|
\begin{array}{cccccc}
A & * & * & * & \cdots & 0\\
0 & A & * & * & \cdots & *\\
* & * & B & 0 &\cdots & 0\\
0 & * & * & B &\cdots & 0\\
\vdots &  \vdots & \vdots &  \vdots & \ddots & \vdots\\
0 & 0 & 0 & 0 & \cdots & B
\end{array}
\right|
\end{eqnarray}
where 
\begin{gather*}
A=(2l_{uv})^{2^{j-4}(j-3)}(d_3\ldots d_{j-1})^{2^{j-3}}+\ldots;\\
B=2l_{uv}d_j^2D_{j+1}^2+\ldots.
\end{gather*}
The dots stand for terms that have smaller total degree with respect to $d_3,\ldots,d_{j-1}$ for~$A$ and with respect to~$d_j$ and~$D_{j+1}$ for~$B$.  Let us prove  property~($F_j$-3). The variable~$d_{j}$ does not enter the first two rows of determinant~\eqref{eq_detAB}. It follows from (CM${}_j$-4) that  the degree  with respect to~$d_j$ of every non-diagonal entry is not greater than~$1$. Hence the power~$d_{j}^{2^{j-2}}$ can appear in the product~$A^2B^{2^{j-3}}$ of the diagonal entries of determinant~\eqref{eq_detAB} only. Besides, it follows from \mbox{($F_{j-1}$-2)} that $D_{j+1}$ does not enter into~$A$. Therefore, the coefficient of the monomial~$(d_3\ldots d_jD_{j+1})^{2^{j-2}}$ in~$F_j$ is equal to
$$
\left((2l_{uv})^{2^{j-4}(j-3)}\right)^2(2l_{uv})^{2^{j-3}}=(2l_{uv})^{2^{j-3}(j-2)}.
$$
Properties~($F_j$-1) and~($F_j$-2) can be obtained from properties ($F_{j-1}$-1), ($F_{j-1}$-2), (CM${}_j$-1), and~(CM${}_j$-2)  by a standard computation of degrees of the resultant.
Now we shall prove property~($F_j$-4). The variables~$d_3$ and~$D_3$ enter the first two rows of determinant~\eqref{eq_detAB} only. If $F_j$ contained a monomial divisible by~$d_3^{2^{j-2}}D_3$, then the product of some two entries of the first two rows of~\eqref{eq_detAB} would be divisible by~$d_3^{2^{j-2}}D_3$. Then one of these two entries would be divisible either by~$d_3^{2^{j-3}}D_3$, or by~$d_3^{2^{j-3}+1}$. This is impossible. Hence $F_j$ contains no monomial divisible by~$d_3^{2^{j-2}}D_3$.

Since~$D_r=l_{w_1w_r}$ is a squared edge length of~$K$, we can omit~$D_r$ in the list of arguments of~$F_{r-2}$. 
Next, we shall exclude successively the variables $d_3,\ldots,d_{r-2}$ from equation~(\ref{eq_Frel}${}_{r-2}$) and equations~(\ref{eq_Z(j)}${}_3$)--(\ref{eq_Z(j)}${}_{r-2}$). 
The equation obtained after excluding~$d_j$ will have the form
\begin{equation*}\refstepcounter{equation}\label{eq_Grel}
G_j(V_Z,l_K,D_3,D_{r-1},d_{j+1},\ldots ,d_{r-2})=0.\eqno{(\theequation_j)}
\end{equation*}
Equation~(\ref{eq_Grel}${}_2$) is equation~(\ref{eq_Frel}${}_{r-2}$) divided by~$2^{r-5}(r-4)$. Equation~(\ref{eq_Grel}${}_j$) is the resultant of  equations~(\ref{eq_Grel}${}_{j-1}$) and~(\ref{eq_Z(j)}${}_j$) with respect to the variable~$d_j$, $j=3,\ldots,r-2$. We need the following properties of the polynomials~$G_j$, $j\ge 3$. There exists a positive integer~$T_j$ such that
\begin{description}
\item[\normalfont($G_j$-1)] The degree of~$G_j$ with respect to each variable~$d_i$ is equal to~$T_j$, $i=j+1,\ldots,r-2$; 
\item[\normalfont($G_j$-2)] The total degree of~$G_j$ with respect to~$D_3$ and $D_{r-1}$ is equal to~$T_j$;
\item[\normalfont($G_j$-3)] If~$M_j$ is the degree of~$G_j$ with respect to~$V_Z$, then the coefficient of the monomial~$(d_{j+1}\ldots d_{r-2}D_{r-1})^{T_j}V_Z^{M_j}$ in~$G_j$ is equal to~$l_{uv}^{S_j}$ for some $S_j>0$;
\item[\normalfont($G_j$-4)] $G_j$ contains no monomial divisible by~$D_3V_Z^{M_j}$.
\end{description}
Notice that $G_2=\frac{1}{2^{r-5}(r-4)}F_{r-2}$ and  properties~($G_2$-1)--($G_2$-3) hold for $M_2=0$, $S_2=2^{r-5}(r-4)$, and $T_2=2^{r-4}$, nevertheless, property~($G_2$-4) does not hold. We shall prove inductively properties~($G_j$-1)--($G_j$-4) for~$j\ge 3$. Assume them for~$j-1$ (except property~($G_2$-4) if~$j=3$) and prove them for~$j$.
The resultant of~(\ref{eq_Grel}${}_{j-1}$) and~(\ref{eq_Z(j)}${}_j$) with respect to the variable~$d_j$ is a $(k_j+T_{j-1})\times (k_j+T_{j-1})$-determinant
\begin{eqnarray}\label{eq_detAB2}
\begin{vmatrix}
A & * & \hdotsfor{5} & 0\\
0 & A & \hdotsfor{5} & 0\\
\vdots & \vdots & \ddots & \vdots & \vdots & \vdots &\ddots &\vdots\\
0 & 0 & \cdots & A & * & * & \cdots & *\\
* & * & \cdots & * & B & 0 &\cdots & 0\\
0 & * & \cdots & * & * & B &\cdots & 0\\
\vdots &  \vdots & \vdots &  \vdots & \vdots & \vdots & \ddots & \vdots\\
0 & 0 & \hdotsfor{5} & B
\end{vmatrix}
\begin{array}{l}
\left.
\vphantom{\begin{vmatrix}
A\\ A \\ \vdots \\ A
\end{vmatrix}}
\right\}k_j\\
\left.
\vphantom{\begin{vmatrix}
A\\ A \\ \vdots \\ A
\end{vmatrix}}
\right\}T_{j-1}
\end{array}
\end{eqnarray}
where $k_j$ is the degree of the left-hand side of~(\ref{eq_Z(j)}${}_j$) with respect to~$d_j$,
\begin{gather*}
A=l_{uv}^{S_{j-1}}(d_{j+1}\ldots d_{r-2}D_{r-1})^{T_{j-1}}V_Z^{M_{j-1}}+\ldots,\\
B=l_{uv}^{2s_j}V_Z^{2N_j}+\ldots
\end{gather*}
The dots stand for terms of smaller total degrees with respect to $d_{j+1},\ldots,d_{r-2}$, $D_{r-1}$, and~$V_Z$ for~$A$ and for terms of smaller degrees with respect to~$V_Z$ for~$B$.  The degree with respect to~$V_Z$ of every entry in the first $k_j$ rows of~\eqref{eq_detAB2}  is not greater than~$M_{j-1}$. The degree with respect to~$V_Z$ of every entry in the last $T_{j-1}$ rows of~\eqref{eq_detAB2}  is not greater than~$2N_j$. Hence, the degree of the resultant~$G_j$ with respect to~$V_Z$ is not greater than~$M_j=k_jM_{j-1}+2T_{j-1}N_j$. Besides, the degree with respect to~$V_Z$ of every non-diagonal entry in the last $T_{j-1}$ rows is strictly smaller than~$2N_j$. Therefore, the power~$V_Z^{M_j}$ can appear in the product~$A^{k_j}B^{T_{j-1}}$ of the diagonal entries of~\eqref{eq_detAB2} only. Consequently, the coefficient of the monomial~$(d_{j+1}\ldots d_{r-2}D_{r-1})^{k_jT_{j-1}}V_Z^{M_j}$ in~$G_j$ is equal to
$l_{uv}^{k_jS_{j-1}+2T_{j-1}s_j}$. In particular, we see that the degree of~$G_j$ with respect to~$V_Z$ is indeed equal to~$M_j$. Obviously, the degree of~$G_j$ with respect to each~$d_i$ is equal to~$k_jT_{j-1}$, $i=j+1,\ldots,r-2$, and the total degree of~$G_j$ with respect to~$D_3$ and~$D_{r-1}$ is equal to~$k_jT_{j-1}$. Properties \mbox{($G_j$-1)}--\mbox{($G_j$-3)}  follow.

Now we shall prove properties~($G_j$-4), $j\ge 3$. Since property~($G_2$-4) does not hold, we first consider the case $j=3$ to obtain a base for the induction. Property~\mbox{($F_{r-2}$-4)} yields that the polynomial~$G_2$ contains no monomial divisible by~$d_3^{T_2}D_3$. Hence the variable~$D_3$ does not enter into the polynomial~$A$. Obviously, $D_3$ does not enter into~$B$ too. Since we have already proved that the power $V_Z^{M_3}$ can appear only in the product of the diagonal entries of~\eqref{eq_detAB2}, we obtain that no monomial of~$G_3$ is divisible by~$D_3V_Z^{M_3}$. Now suppose that $j\ge 4$ and property~\mbox{($G_{j-1}$-4)} is proved. The variable~$D_3$ enters the first $k_j$ rows of~\eqref{eq_detAB2} only. The variable~$V_Z$ enters the first $k_j$ rows with degrees not exceeding~$M_{j-1}$ and  the last $T_{j-1}$ rows with degrees not exceeding~$2N_j$. Recall that $M_j=k_jM_{j-1}+2T_{j-1}N_j$. Hence if $G_j$ contained a monomial divisible by~$D_3V_Z^{M_j}$, then some monomial in the  first $k_j$ rows would be divisible by~$D_3V_Z^{M_{j-1}}$, which is impossible. Thus  $G_j$ contains no monomial divisible by~$D_3V_Z^{M_j}$.

For brevity we denote the numbers~$M_{r-2}$, $S_{r-2}$, and~$T_{r-2}$ by~$M$, $S$, and~$T$ respectively. 
Polynomial relation~(\ref{eq_Grel}${}_{r-2}$) can be rewritten as
\begin{equation}\label{eq_SMT}
l_{uv}^SV_Z^MD_{r-1}^T+f(V_Z,l_K,D_{r-1})+g(V_Z,l_K,D_3,D_{r-1})=0,
\end{equation}
where the polynomial~$f$ has degree  with respect to~$V_Z$ not greater than $M$ and degree with respect to~$D_{r-1}$ strictly less than~$T$  and the polynomial~$g$ has degree with respect to~$V_Z$ strictly less than~$M$, and the total degree with respect to~$D_3$ and~$D_{r-1}$ not greater than~$T$. 

Now we notice that the initial situation is symmetric with respect to renumbering the vertices of the simple cycle $w_1,\ldots,w_r$ in the opposite direction. This means that we can repeat the above reasoning exchanging everywhere vertices~$w_j$ with vertices~$w_{r-j+2}$, $j=2,\ldots,r$. First, we successively exclude the variables~$D_{r-2},\ldots,D_4$ from equations~(\ref{eq_CMj}${}_{r-3}$)--(\ref{eq_CMj}${}_{2}$). Second, we successively exclude the variables~$d_{r-3},\ldots,d_{2}$ from the obtained equation and equations~(\ref{eq_Z(j)}${}_{r-3}$)--(\ref{eq_Z(j)}${}_{2}$).
Finally, for some positive integers~$\widetilde M$, $\widetilde S$, and~$\widetilde T$, we obtain a relation
\begin{equation}\label{eq_SMT_tilde}
l_{uv}^{\widetilde S}V_Z^{\widetilde M}D_{3}^{\widetilde T}+\tilde{f}(V_Z,l_K,D_{3})+\tilde g(V_Z,l_K,D_3,D_{r-1})=0,
\end{equation}
where the polynomial~$\tilde f$ has degree  with respect to~$V_Z$ not greater than $\widetilde M$ and degree with respect to~$D_3$ strictly less than~$\widetilde T$  and the polynomial~$\tilde g$ has degree with respect to~$V_Z$ strictly less than~$\widetilde M$, and the total degree with respect to~$D_3$ and~$D_{r-1}$ not greater than~$\widetilde T$.

Consider the resultant of the left-hand sides of equations~\eqref{eq_SMT} and~\eqref{eq_SMT_tilde} with respect to~$D_{r-1}$. It can be easily seen that the obtained relation has form
\begin{equation}\label{eq_ABC}
l_{uv}^AV_Z^BD_3^C+\ldots=0,
\end{equation} 
where $B$ and $C$ are the degrees of the left-hand side with respect to~$V_Z$ and~$D_3$ respectively and dots stand for terms of total degrees with respect to~$V_Z$ and~$D_3$ strictly smaller than~$B+C$. Now notice that $D_3=d_1$. Excluding the variable~$D_3=d_1$ from equations~\eqref{eq_ABC} and~(\ref{eq_Z(j)}${}_1$) we obtain the required relation
$$
l_{uv}^{s}V_Z^{N}+a_1(l_K)V_Z^{N-1}+\ldots+a_{N}(l_K)=0.
$$

\section{Proof of Lemma~\ref{lem_alg}}\label{section_alg}

Let $X\subset\C^r$ be an irreducible affine variety. We say that $X$ is \textit{defined over\/}~$\Q$ if it can be given by polynomial equations with rational coefficients. 
For $F=\Q$ or~$\C$, we denote by $F[X]$ the ring of regular functions on~$X$ with coefficients in~$F$ and by~$F(X)$ the field of rational functions on~$X$ with coefficients in~$F$. We start with the following algebraic lemma. It seems to be standard, however, for the convenience of the reader, we give its proof.

\begin{lem}\label{lem_algebraic}
Let $X\subset\C^r$ be an irreducible affine variety defined over~$\Q$. Let $f_1,\ldots,f_p$ be elements of~$\Q[X]$ such that the subset $\{f_1=\cdots=f_p=0\}\subset X$ has codimension at least~$2$. Let $y$ be an element of an extension~$L$ of the field~$\Q(X)$ such that the elements $f_1^{s_1}y,\ldots,f_p^{s_p}y$ are integral over~$\Q[X]$ for some $s_1,\ldots,s_p\ge 0$. Then the element~$y$ is integral over~$\Q[X]$.
\end{lem}

\begin{proof}
First, suppose that $y\in \Q(X)$.
Let  $\pi:\widetilde{X}\to X$ be the normalization of~$X$. Then the ring $\C[\widetilde{X}]$ is the integral closure of~$\C[X]$ in~$\C(X)$. Hence $y$ is a rational function on~$\widetilde{X}$ such that $f_j^{s_j}y$ are regular functions on~$\widetilde{X}$, $j=1,\ldots,p$. Since~$\widetilde{X}$ is normal, the divisor of poles~$(y)_{\infty}$ of~$y$ on~$\widetilde{X}$ is well defined. Let $Y\subset \widetilde{X}$ be the support of~$(y)_{\infty}$. For every~$j$, the function~$f_j^{s_j}y$ is regular on~$\widetilde{X}$, hence, $Y$ is contained in the subset $\{f_j=0\}\subset \widetilde{X}$.  Since $\pi$ is a finite regular morphism, we obtain that the codimension of the subset $\{f_1=\cdots=f_p=0\}\subset\widetilde X$ is the same as the codimension of the subset $\{f_1=\cdots=f_p=0\}\subset X$, that is, is at least~$2$. Therefore, $\codim(Y)\ge 2$. Since $Y$ is the support of the divisor~$(y)_{\infty}$, it follows that $(y)_{\infty}=0$. Hence $y$ is a regular function on~$\widetilde{X}$. Therefore $y$  is integral over~$\C[X]$. 
Since $\C[X]=\Q[X]\otimes\C$ and $\C(X)\supset\Q(X)\otimes\C$, it can be easily seen that every place $\varphi:\Q(X)\to F$ that is finite on~$\Q[X]$ extends to a place $\Phi:\C(X)\to F'$ that is finite on~$\C[X]$, where $F'$ is an extension of~$F$. Hence an element of $\Q(X)$ is integral over~$\Q[X]$ whenever it is integral over~$\C[X]$. 
Thus $y$ is integral over~$\Q$.

Now let $y$ be an element of an arbitrary extension $L\supset\Q(X)$. At least one of the elements $f_j$ is nonzero. Hence the element~$y$ is algebraic over~$\Q(X)$. Then there is a unique irreducible monic polynomial $H\in \Q(X)[t]$ such that $H(y)=0$. Suppose,
$$
H(t)=t^N+c_1t^{N-1}+\ldots+c_N
,\qquad c_i\in \Q(X).$$ 
It is well known that $y$ is integral over~$\Q[X]$ if and only if all coefficients $c_1,\ldots,c_k$ are integral over~$\Q[X]$. The element~$f_j^{s_j}y$ is a root of the irreducible  polynomial
\begin{equation*}
H_j(t)=t^N+f_j^{s_j}c_1t^{N-1}+f_j^{2s_j}c_2t^{N-2}+\ldots +f_j^{Ns_j}c_N.
\end{equation*}
Since the elements $f_j^{s_j}y$ are integral over~$\Q[X]$, $j=1,\ldots,p$, we obtain that the elements $f_j^{ir}c_i$ are integral over~$\Q[X]$, $i=1,\ldots,N$, $j=1,\ldots,p$. But we have already proved the statement of the lemma for the elements~$c_i\in \Q(X)$. Thus the elements $c_1,\ldots,c_N$ are integral over~$\Q[X]$. Hence $y$ is integral over~$\Q[X]$.
\end{proof}

Now let~$K$ be a pure $(n-1)$-dimensional simplicial complex with $m$ vertices and $r$ edges,  $n\ge 2$.
We consider the $mn$-dimensional affine space~$\C^{mn}$ with coordinates~$x_{u,i}$, where $u$ runs over all vertices of~$K$ and $i$ runs from~$1$ to~$n$, and the $r$-dimensional affine space~$\C^r$ with coordinates~$\ell_{uv}=\ell_{vu}$, where $[uv]$ runs over all edges of~$K$. (We denote the independent coordinates in~$\C^r$ by~$\ell_{uv}$ to avoid the confusion with elements~$l_{uv}\in\Q[x_K]$.) Let $\ell_K$ be the set of all variables~$\ell_{uv}$. Let  $h:\C^{mn}\to\C^{r}$ be the regular morphism given by
$$
\ell_{uv}=\sum_{i=1}^n(x_{u,i}-x_{v,i})^2.
$$
The image of~$h$ is the set $\LL_{K}$ of  vectors~$l(P)\in\C^r$ for all polyhedra $P:K\to\C^n$. Let $X_K$ be the Zariski closure of~$\LL_K$ in~$\C^r$. Then $X_K$ is an irreducible affine variety and $h$ is the composition
of the dominant morphism~$h_1:\C^{mn}\to X_K$ and the injective morphism~$h_2:X_K\to\C^r$. Since the morphism~$h$ is defined over~$\Q$, we see that the variety~$X_K$ is defined over~$\Q$ and $\Q[X_K]$ is the $\Q$-subalgebra $R_K\subset\Q[x_K]$ generated by the elements $h^*(\ell_{uv})=l_{uv}$. 

\begin{lem}\label{lem_codim}
Let $u$ be a vertex of~$K$ and let $v_1,\ldots,v_p$ be all vertices of~$K$ joined by edges with~$u$. Suppose, $p\ge 2$. Then the subset $Y\subset X_K$ given by the equations $l_{uv_1}=\cdots=l_{uv_p}=0$ has codimension at least~$2$.  
\end{lem}

\begin{remark}\label{remark_codim}
Lemma~\ref{lem_codim} is not as trivial as it may seem to be. Actually, it is quite obvious that the intersection~$Y\cap\LL_K$ has codimension~$p$, since we can move each of the points~$P(v_j)$ so that to change arbitrarily the square of length~$l_{uv_j}$ without changing other squares of lengths~$l_{uv_i}$, $i\ne j$. However, it is much harder to deal with points added to~$\LL_K$ under the Zariski closure. For example, the author does not know whether the subset of~$K$ given by $2$ equations~$l_{uv_i}=l_{uv_j}=0$ always has codimension~$2$ or not.
\end{remark}

For the proof of Lemma~\ref{lem_alg} we need the case $n=4$ only. Since $K$ is pure, it must contain at least one $3$-dimensional simplex containing the vertex~$u$. Hence we always have $p\ge 3$. 
Thus
Lemma~\ref{lem_alg} follows immediately from Lemmas~\ref{lem_algebraic} and~\ref{lem_codim}.
The proof of Lemma~\ref{lem_codim} is based on the following two lemmas.

\begin{lem}\label{lem_estimate}
Suppose $P:K\to\C^n$ is a polyhedron and $\varepsilon=\max_{1\le j\le p}|l_{uv_j}(P)|$. Then there exist a number $j$, $1\le j\le p$, and a polyhedron $P':K\to\C^n$ such that $l_{uv_j}(P')=0$ and $|l_{w_1w_2}(P')-l_{w_1w_2}(P)|\le3\varepsilon$ for every edge $[w_1w_2]$ of~$K$.  
\end{lem}

\begin{proof}
If for some~$j$, $l_{uv_j}(P)=0$, then $P$ itself is the required polyhedron. So we may assume that $l_{uv_j}(P)\ne 0$, $j=1,\ldots,p$.

The squares of lengths~$l_{w_1w_2}(P)$ are defined by means of the standard bilinear scalar product in~$\C^n$ given by
$$
(\xi,\eta)=\xi^1\eta^1+\cdots+\xi^n\eta^n,
$$
where $\xi=(\xi^1,\ldots,\xi^n)$ and $\eta=(\eta^1,\ldots,\eta^n)$.
Alongside with this bilinear scalar product, we consider the standard Hermitian scalar product
$$
\langle\xi,\eta\rangle=(\bar\xi,\eta)=\bar\xi^1\eta^1+\cdots+\bar\xi^n\eta^n.
$$ 
By $|\xi|$ we define the Hermitian length~$\sqrt{\langle\xi,\xi\rangle}$ of a vector~$\xi$.

We put $\xi_i=P(v_i)-P(u)$, $i=1,\ldots,p$. Let $h=|\xi_j|$ be the maximum of the Hermitian lengths $|\xi_1|,\ldots,|\xi_p|$ and put $\lambda=l_{uv_j}(P)=(\xi_j,\xi_j)$. Let $\varkappa$ and $\varkappa'$ be the roots of the equation
$$
\bar\lambda t^2-2h^2t+\lambda=0
$$
such that $|\varkappa|\le |\varkappa'|$. Then $|\varkappa\varkappa'|=1$ and
$$
|\varkappa|+|\varkappa'|\ge \frac{2h^2}{|\lambda|}\ge \frac{2h^2}{\varepsilon}
$$
It easily follows that $|\varkappa|\le\frac{\varepsilon}{h^2}$. Consider the polyhedron $P':K\to\C^n$ such that $P'(u)=P(u)+\varkappa\bar\xi_j$ and $P'(w)=P(w)$ for every vertex~$w\ne v$. First, we have
$$
l_{uv_j}(P')=(\xi_j-\varkappa\bar\xi_j,\xi_j-\varkappa\bar\xi_j)=\lambda-2h^2\varkappa+\bar\lambda \varkappa^2=0.
$$
Second, for every $i\ne j$, we have
$$
l_{uv_i}(P')-l_{uv_i}(P)=(\xi_i-\varkappa\bar\xi_j,\xi_i-\varkappa\bar\xi_j)-(\xi_i,\xi_i)=-2\varkappa(\bar\xi_j,\xi_i)+\varkappa^2(\bar\xi_j,\bar\xi_j).
$$
Hence,
$$
|l_{uv_i}(P')-l_{uv_i}(P)|\le 2|\varkappa| |\xi_j||\xi_i|+|\varkappa|^2|\xi_j|^2\le 3\varepsilon.
$$
Finally, for every edge~$[w_1w_2]$ such that $w_1\ne u$ and $w_2\ne u$, we have $l_{w_1w_2}(P')=l_{w_1w_2}(P)$.
\end{proof}

\begin{lem}\label{lem_sdvig}
Let $l=(l_{w_1w_2})$ be a point in~$X_K$ such that $l_{uv_i}=0$ for $i=1,\ldots,p$. Let $\varepsilon$ be a positive number. Then there exists a polyhedron~$P:K\to\C^n$ such that
\begin{itemize}
\item $|l_{w_1w_2}(P) -l_{w_1w_2}|\le\varepsilon$ for every edge $[w_1w_2]\in K$;
\item there exists $j$, $1\le j\le p$, such that $l_{uv_j}(P)=0$; 
\item there exists $i$, $1\le i\le p$, such that $l_{uv_i}(P)\ne 0$.
\end{itemize} 
\end{lem}

\begin{proof}
Since $h_1:\C^{mn}\to X_K$ is a dominant morphism of irreducible affine varieties, we see that the set $\LL_K=h_1(\C^{mn})$ contains a Zariski open subset $U\subset X_K$. Hence~$\LL_K$ is dense in the analytic topology on~$X_K$. (The analytic topology is the topology induced by the Hermitian metric in~$\C^r$.) Therefore there exists a point $l'=(l'_{w_1w_2})\in \LL_K$ such that $|l'_{w_1w_2}-l_{w_1w_2}|\le \frac{\varepsilon}{5}$ for every edge~$[w_1w_2]$. Let $P':K\to\C^n$ be a polyhedron such that $l(P')=l'$. Since $l_{uv_i}=0$, we see that $|l_{uv_i}(P')|\le \frac{\varepsilon}{5}$, $i=1,\ldots,p$. By Lemma~\ref{lem_estimate}, there exists a polyhedron $P'':K\to\C^n$ such that $l_{uv_j}(P'')=0$ and $|l_{w_1w_2}(P'')-l_{w_1w_2}(P')|\le\frac{3\varepsilon}{5}$ for every edge $[w_1w_2]$ of~$K$. Then $|l_{w_1w_2}(P'')-l_{w_1w_2}|\le\frac{4\varepsilon}{5}$. If there exists $i$, $1\le i\le p$, such that $l_{uv_i}(P'')\ne 0$, then $P''$ is the required polyhedron. Now assume that $l_{uv_i}(P'')=0$ for every $i$. Choose an arbitrary~$i$ such that $1\le i\le p$ and~$i\ne j$. Consider a new polyhedron $P:K\to\C^n$ such that $P(w)=P''(w)$ for every~$w\ne v_i$ and the point $P(v_i)$ is obtained from the point~$P''(v_i)$ by the shift along a vector~$\eta\in\C^n$. For a sufficiently small vector $\eta$ in general position, we shall obtain that  $l_{uv_i}(P)\ne 0$ and $|l_{v_iw}(P)-l_{v_iw}(P'')|\le\frac{\varepsilon}{5}$ for every vertex~$w$ connected by an edge with~$v_i$. Then $|l_{w_1w_2}(P)-l_{w_1w_2}|\le\varepsilon$ for every edge~$[w_1w_2]$. Besides, $l_{uv_j}(P)=l_{uv_j}(P'')=0$. Hence $P$ is the required polyhedron. 
\end{proof}

\begin{proof}[Proof of Lemma~\ref{lem_codim}]
By $Y_i$ denote the subset of~$X_K$ given by the equation~$l_{uv_i}=0$. Let $Y_i=Y_{i,1}\cup\ldots\cup Y_{i,q_i}$ be the decomposition into irreducible components. Assume that the set $Y=Y_1\cap\ldots\cap Y_p$ contains an irreducible component~$C$ whose codimension in~$X_K$ is equal to~$1$. Then, for every~$i$, $C$ is an irreducible component of~$Y_i$.
Without loss of generality we may assume that $Y_{i,1}=C$, $i=1,\ldots,p$. Then there exists a point $l=(l_{w_1w_2})\in C$ such that $l$ does not belong to the union of all components~$Y_{i,k}$ such that $k>0$. We have $l_{uv_i}=0$, $i=1,\ldots,p$. For $\varepsilon>0$, we denote by~$U_{\varepsilon}(l)$ the set of all points $l'=(l'_{w_1w_2})\in X_K$ such that $|l'_{w_1w_2}-l_{w_1w_2}|\le\varepsilon$ for every edge~$[w_1w_2]\in K$. Since the sets~$Y_{i,k}$ are closed, we see that, for a sufficiently small~$\varepsilon$, the set $U_{\varepsilon}(l)$ does not intersect the union of all~$Y_{i,k}$ such that $k>0$. Therefore, $U_{\varepsilon}(l)\cap (Y_1\cup\ldots\cup Y_p)=U_{\varepsilon}(l)\cap C$. However, by Lemma~\ref{lem_sdvig}, there exists a point $l'\in U_{\varepsilon}(l)$ such that $l'\in Y_j$ and $l'\notin Y_i$ for some~$j$ and~$i$. Hence $l'\in Y_1\cup\ldots\cup Y_p$, but $l'\notin C$. This yields a contradiction, which completes the proof of the lemma.
\end{proof}

\section{Towards higher dimensions}\label{section_high}

We do not know whether the analogues of Theorems~\ref{theorem_main3} and~\ref{theorem_main4} (or, equivalently, the analogues of Theorem~\ref{theorem_alg4}) hold in dimensions greater than~$4$. However, in this section we shall make some progress towards the generalization of our proof to higher dimensions and we shall show which problems arise on this way.

Consider a finite pure $(n-1)$-dimensional simplicial complex~$K$. By $m(K)$ we denote the number of vertices of~$K$. For any $i$-simplex $\tau\in K$ we denote by~$m(\tau,K)$ the number of $(i+1)$-simplices of~$K$ containing~$\tau$. An $(n-4)$-flag of~$K$ is the sequence~$\T$ of simplexes $\tau^0\subset\tau^1\subset\cdots\subset\tau^{n-4}$ of the complex~$K$ such that $\dim\tau^i=i$. By~$\mathcal{F}(K)$ we denote the set of all $(n-4)$-flags of~$K$. We  assign the integral vector 
$$
\m(\T,K)=(m(K),m(\tau^0,K),m(\tau^1,K),\ldots,m(\tau^{n-4},K))
$$
to any $(n-4)$-flag~$\T\in\mathcal{F}(K)$.
Let~$\prec$ be the lexicographic ordering on the set of all vectors $(m_{-1},m_0,\ldots,m_{n-4})\in\Z^{n-2}_{\ge 0}$. This means that 
$$
(m_{-1},m_0,\ldots,m_{n-4})\prec (m_{-1}',m_0',\ldots,m_{n-4}')
$$
whenever there exists $j$ such that   $m_j<m_j'$ and $m_i=m_i'$ for all $i<j$. 
We put 
$$
\m(K)=\min_{\T\in\mathcal{F}(K)}\m(\T,K),
$$
where the minimum is understood with respect to~$\prec$. The vector $\m(K)$ is well defined for any nonempty pure $(n-1)$-dimensional complex~$K$ because any such complex necessarily contains at least one $(n-4)$-flag. An $(n-4)$-flag~$\T\in\mathcal F(K)$ is called \textit{minimal} if $\m(\T,K)=\m(K)$. 
It is convenient to allow the complex~$K$ to be empty and to suppose that $\m(\emptyset)=(0,\ldots,0)$.
It is well known that the ordering~$\prec$ on~$\Z_{\ge 0}^{n-2}$ is complete. Hence we can try to prove the analogue of Theorem~\ref{theorem_alg4} by the induction on~$\m(K)$. 

For every simplex $\sigma\in K$, let $\D_{\sigma}\in\Q[x_K]$ be the Cayley--Menger determinant of the vertices of~$\sigma$.
The following lemma is an analogue of Lemma~\ref{lem_step} for an arbitrary dimension.

\begin{lem}
\label{lem_steph}
Suppose, $n\ge 3$ and $\m\in\Z_{\ge 0}^{n-2}$. Assume that, for any $(n-1)$-cycle~$\underline{Z}$ such that $\m(\supp(\underline{Z}))\prec\m$, the element~$V_{\underline{Z}}$ is integral over the ring~$R_{\supp(\underline{Z})}$. Let $Z$ be an $(n-1)$-cycle with support~$K$ such that $\m(K)=\m$. Let $\T=\{\tau^0\subset\cdots\subset\tau^{n-4}\}$ be a minimal $(n-4)$-flag of~$K$ and let $\sigma$ be an $(n-3)$-dimensional simplex containing~$\tau^{n-4}$. Then for some $s\ge 0$, the element~$\D_{\sigma}^sV_Z$ is integral over~$R_K$. 
\end{lem}

The proof of Lemma~\ref{lem_steph} for an arbitrary~$n\ge 3$ is completely similar to the proof of Lemma~\ref{lem_step}. Again, we can give non-constructive and constructive proofs. In the non-constructive proof the only difference is that we should replace the condition $\varphi(l_{uv})\ne 0$ in Lemmas~\ref{lem_place1} and~\ref{lem_place2} by the condition~$\varphi(\D_{\sigma})\ne 0$. In the constructive proof we should replace everywhere $2l_{uv}$ by~$\D_{\sigma}$. The matter is that the coefficients of~$d_1^2d_2^2$ in~\eqref{eq_CM_4d1d2} and of~$d_j^2D_{j+1}^2$ in~(\ref{eq_CMj}${}_j$) become equal to~$\D_{\sigma}$.

If $n=3$, then $\D_{\sigma}=1$, hence, the assertion of Lemma~\ref{lem_steph} is that $V_Z$ is integral over~$R_K$. Therefore Theorem~\ref{theorem_main3} follows inductively and no analogue of Lemma~\ref{lem_alg} is needed. Notice that the obtained proof of Theorem~\ref{theorem_main3} is simpler than the proofs known before because it does not require the induction on the genus of~$K$.
If $n=4$, then $\sigma$ is an edge~$[uv]$ and $\D_{\sigma}=2l_{uv}$. Therefore, Lemma~\ref{lem_steph} turns into Lemma~\ref{lem_step}. 
The multidimensional analogue of Theorems~\ref{theorem_main3} and~\ref{theorem_main4} would follow inductively if we proved the following analogue of Lemma~\ref{lem_alg}: 

\textit{Let $Z$ be an $(n-1)$-cycle with support~$K$. Let $\tau$ be an $(n-4)$-dimensional simplex of~$K$ and let $\sigma_1,\ldots,\sigma_p$ be all $(n-3)$-dimensional simplices of~$K$ containing~$\tau$. Assume that, for every~$j$, there exists a nonnegative integer~$s_j$ such that the element~$\D_{\sigma_j}^{s_j}V_Z$ is integral over~$R_K$. Then the element~$V_Z$ is also integral over~$R_K$.} 

Nevertheless, we do not know whether this assertion is true or not. By Lemma~\ref{lem_algebraic}, this assertion would follow if the answer to the following question were positive.

\begin{quest}
Let $K$ be a finite pure $(n-1)$-dimensional simplicial complex with $r$ edges such that there exists an $(n-1)$-dimensional cycle with support~$K$. Let $\tau$ be an $(n-4)$-dimensional simplex of~$K$ and let $\sigma_1,\ldots,\sigma_p$ be all $(n-3)$-dimensional simplices of~$K$ containing~$\tau$. Does the subset $\{\D_{\sigma_1}=\cdots=\D_{\sigma_p}=0\}\subset X_K$ necessarily has codimension at least~$2$?
\end{quest}

The positive answer to this question for some~$n\ge 5$ would imply the analogue of Theorems~\ref{theorem_main3} and~\ref{theorem_main4} for this dimension~$n$.
As it was already mentioned in section~\ref{section_alg} (see Remark~\ref{remark_codim}) it is easy to show that the codimension of the subset of~$\LL_K$ given by $\D_{\sigma_1}=\cdots=\D_{\sigma_p}=0$ is at least $p$, which is greater than~$1$. However, we cannot prove the same in the Zariski closure of~$\LL_K$. The matter is that we cannot obtain a proper analogue of Lemma~\ref{lem_estimate}.

\end{document}